\documentclass[12pt]{article}    
\usepackage{setspace}
\usepackage{hyperref}
\usepackage{graphics}
\usepackage{epsfig}
 
\usepackage{amsmath}
\usepackage{mathrsfs}

\usepackage{enumitem}

\usepackage{amssymb}
\usepackage{setspace}
\usepackage{hyperref}
\usepackage{graphics}
\usepackage{epsfig}
\usepackage{url}
\usepackage{amsmath}
\usepackage{bm}
\usepackage{amsfonts}
\usepackage{amsmath}
\usepackage{mathrsfs}
\usepackage{multirow}
\numberwithin{equation}{section}

\usepackage{color}
\usepackage{hyperref}
\usepackage{graphicx}
\usepackage{graphics}

\usepackage{titlesec}
\usepackage{mathtools}
\usepackage{amsfonts}

\usepackage{caption}
\usepackage{subcaption}
\numberwithin{equation}{section}

\def \R {{\rm I\kern -2.2pt R\hskip 1pt}}

\newcommand{\qed}{\rule{0.5em}{1.5ex}}

\newtheorem{proposition}{Proposition}

\newtheorem{remark}{Remark}
\newtheorem{theorem}{Theorem} 
\newtheorem{corollary}{Corollary} 
\newenvironment{proof}{\noindent {\bf Proof.\nopagebreak}}{~\qed}

\begin{document}

\begin{center} { \large \sc  Some power function distribution processes}\\
\vskip 0.1in {\bf Barry C. Arnold}\footnote{barnold@ucr.edu}
\\
\vskip 0.1in 
Department of Statistics, University of California, Riverside, USA. \\

\vskip 0.1in {\bf  Sachin Sachdeva \footnote{sachinstats007@gmail.com} and B.G. Manjunath \footnote{bgm@uohyd.ac.in}}
\\
\vskip 0.1in 
 School of Mathematics and Statistics, University of Hyderabad, Hyderabad, India. \\

\vspace{.3in}

April  25, 2023

\end{center}

 \begin{abstract}
 	It is known that all the proportional reversed hazard (PRH) processes can be derived by a marginal transformation applied to a power function distribution (PFD) process. Kundu \cite{kd22} investigated PRH processes that can be viewed as being obtained by marginal transformations applied to a particular PFD
 	process that will be described and investigated and will be called a Kundu process. In the present note, in addition to studying the Kundu process, we introduce a new PFD process having  Markovian  and stationarity properties.  We discuss distributional features of such processes, explore inferential aspects and include an example of applications of the PFD processes to  real-life data.
 \end{abstract}
 
 
 \bigskip

 \bigskip
 
 \section{A recent construction and a proposed alternative}
 
 A recent article by Kundu \cite{kd22} discussed certain proportional reversed hazard (PRH) processes that can be obtained by a marginal transformation applied to a particular power function distribution (PFD) process. We will describe that PFD process and will call it a Kundu process. In addition to the Kundu process we will investigate a new stationary stochastic process with power function distribution marginals. It, like the Kundu process, will rely crucialy on the fact that maxima of independent power function variables will themselves have power function distributions. The new process has an autoreressive flavor involving power function innovation variables, but with maximization playing the role usually played by addition. In addition some attention will be focussed on higher order versions of these processes
 
 The reversed hazard function has been used extensively in analyzing lifetime data, in Engineering for life testing of machineries and in forensic science, c.f for example in  Crescenzo \cite{c00} and Gupta and Gupta \cite{gg07}. Many proportional reversed hazard models have received significat attention in the literature.  For example, exponentiated exponential by Gupta and Kundu \cite{gk99}, exponentiated Pareto by Shawky and Abu-Zinadah \cite{sa08}, exponentiated gamma by Gupta,Gupta and Gupta \cite{ggg98}  etc. It is then to be expected that PRH processes will be useful models for certailn engineering time series data settings. 
 
 Since the stochastic behavior of a PRH process is essentially determind by the properties of the underlying PFD process, it will be useful to investigate distributional prperties of the PFD processes discussed in this paper. In a separate repot, attention will be given to some of the PRH models obtainable from them via marginal transformations.
 \\
 \section{The Kundu process}
 The Kundu \cite{kd22} PFD process is defined as follows. Here and subsequently we will consider an i.i.d. 
 sequence, $\left\{U_n \right\} $, of $uniform(0,1)$ random variables. Define the process 
 $\left\{ X_n \right\} $ by
 \begin{equation}\label{PFDP}
 	X_n= \max \left\{U_n^{1/\alpha},U_{n-1}^{1/\beta} \right\},
 \end{equation}
 where $\alpha, \beta >0$. We refer to Figures \ref{fig1}--\ref{fig2} for examples of simulated sample paths of steps  $n=10$ and $n=100$.

 \begin{figure} [h!]
 	\begin{subfigure}{8cm}
 		\centering
 		\includegraphics[width=8cm]{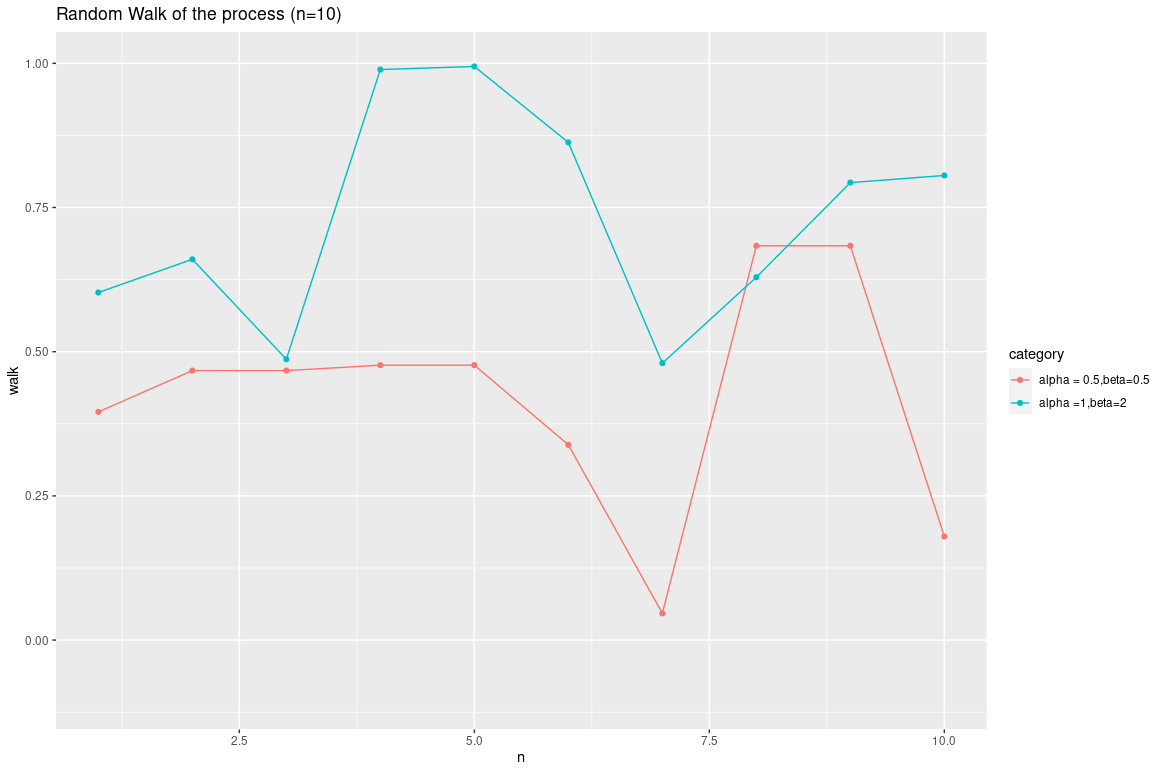}
 		\caption{$n=10$} 
 	\end{subfigure}
 	\begin{subfigure}{8cm}
 		\centering
 		\includegraphics[width=8cm]{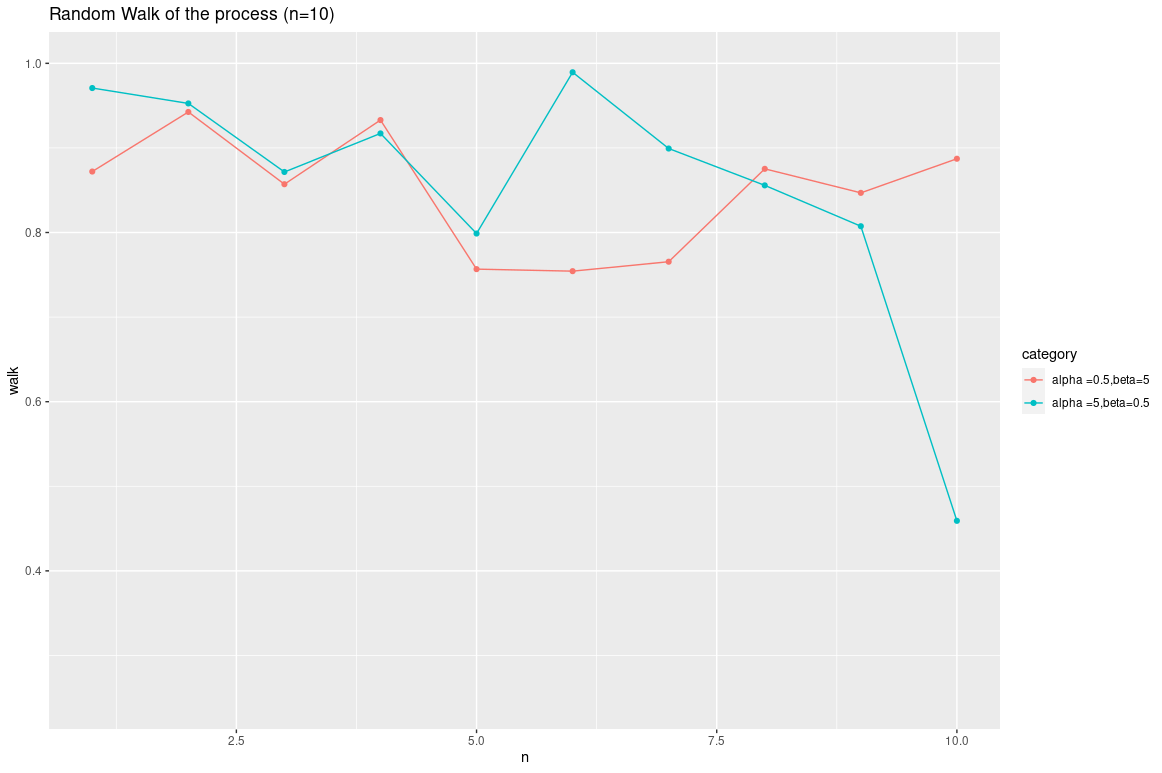}
 		\caption{$n=10$} 
 	\end{subfigure}
 	\caption{Sample Path (number of steps is $10$)}
 	\label{fig1}
 \end{figure}

 \begin{figure} \label{fig2} 
 	\begin{subfigure}{8cm}
 		\centering
 		\includegraphics[width=8cm]{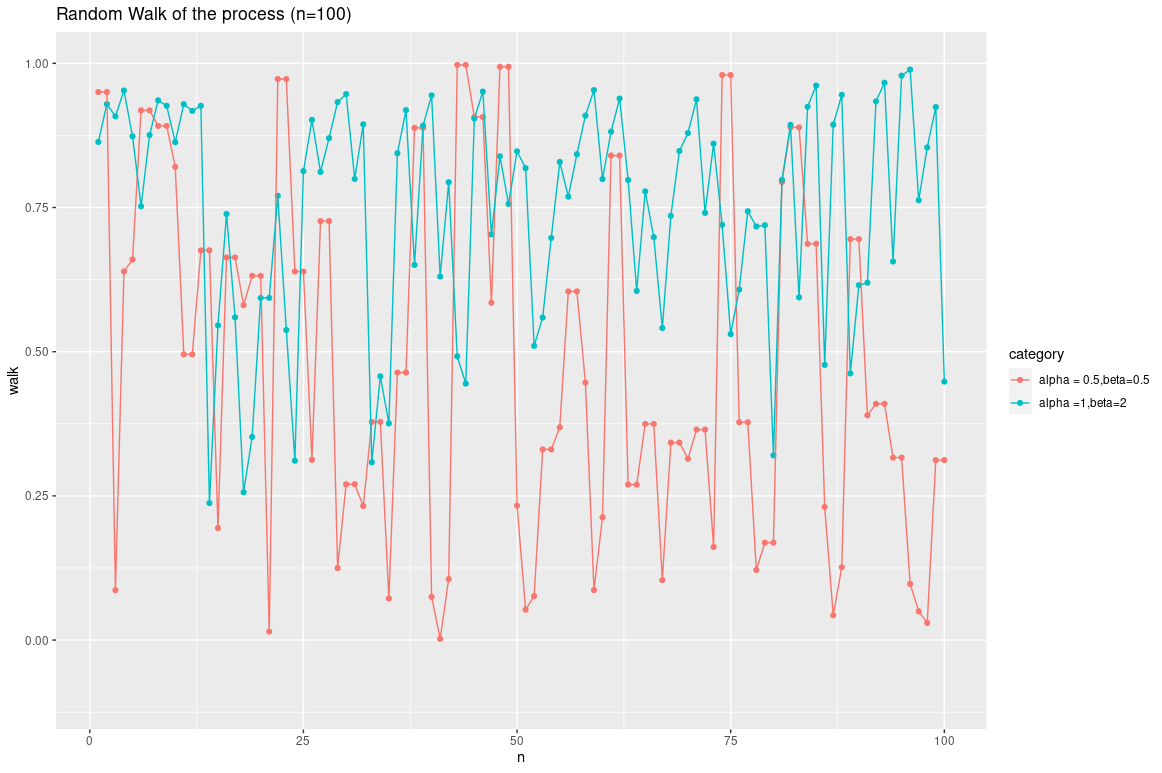}
 		\caption{$n=100$}
 	\end{subfigure}
 	\begin{subfigure}{8cm}
 		\centering
 		\includegraphics[width=8cm]{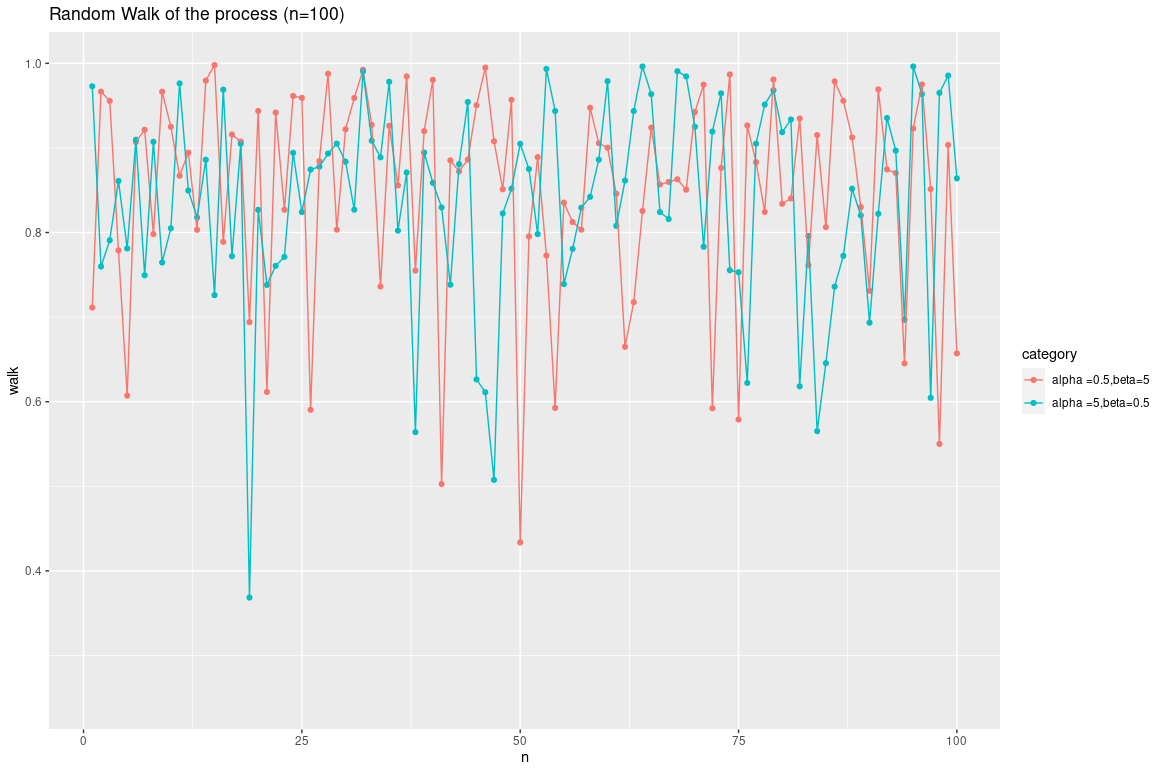}
 		\caption{$n=100$} 
 	\end{subfigure}
 	\caption{Sample Path (number of steps is $100$)}
 	\label{fig2}
 \end{figure}

 In the following subsection we derive some stochastic properties of the process defined in $(2.1)$.
 \subsection{Stochastic properties of the process defined in $(2.1)$}    
 
 \begin{theorem}
 	If the sequence of random variables $\{X_n\}$ is as defined in equation $(2.1)$, then the following statements are true:
 	\begin{enumerate} [label=(\alph*)]
 		\item For each $n$, $\{X_n\}$ has a PFD($\alpha + \beta$) distribution.
 		\item The joint distribution function of $(X_{n-1},X_n)$  is $x^{\alpha}_n x^{\beta}_{n-1} \min\{x^{\alpha}_n, x_{n-1}^{\beta}\}$
 		\item $\{X_n\}$ is a stationary non-Markovian process. 
 	\end{enumerate}
 \end{theorem}
 
 \begin{proof}
 	To prove part $(a)$,  we introduce convenient new notation for $X_n$ and $X_{n-1}$ in the process defined in equation $(2.1)$ as follows
 	\begin{eqnarray*}
 		X_n &=&\max\{U^{1/\alpha},V^{1/\beta} \}\\
 		X_{n-1} &=&\max\{V^{1/\alpha},W^{1/\beta} \}\
 	\end{eqnarray*}
 	where $U,V$ and $W$ are i.i.d. $uniform(0,1)$ random variables. 
 	
 	\bigskip
 	Now, the distribution function of $\{ X_n \}$  is
 	\begin{eqnarray}
 		F(x_n) = P(X_n \leq x) &=& P \Big( \max\{U^{1/\alpha},V^{1/\beta} \} \leq x \Big)  \nonumber \\
 		&=& P\Big( U^{1/\alpha} \leq x, V^{1/\beta} \leq x \Big)  \nonumber \\
 		&=& x^{\alpha + \beta}.
 	\end{eqnarray}
 	
 	For the proof of part $(b)$ which identifies the joint distribtion function of  $X_n$ and $X_{n-1}$, we argue as follows.
 	\begin{eqnarray}
 		F_{X_n,X_{n-1}}(x_n,x_{n-1}) &=& P(X_n \leq x_n, X_{n-1} \leq x_{n-1})  \nonumber \\
 		&=& P \Big( \max\{U^{1/\alpha},V^{1/\beta} \} \leq x_n, \max\{V^{1/\alpha},W^{1/\beta} \} \leq x_{n-1} \Big) \nonumber  \\
 		&=& P \Big( U^{1/\alpha} \leq x_n, V^{1/\beta} \leq x_n, V^{1/\alpha} \leq x_{n-1}, W^{1/\beta} \leq x_{n-1}   \Big) \nonumber \\
 		&=& P\Big( U \leq x^\alpha_n, V \leq min\{x^{\alpha}_n, x_{n-1}^{\beta}\},  W \leq x_{n-1}^\beta \Big) \nonumber \\
 		&=& x^{\alpha}_n x^{\beta}_{n-1} \min\{x^{\alpha}_n, x_{n-1}^{\beta}\}.
 	\end{eqnarray}
 	It remains to prove Part $(c)$.  Stationarity is trivial from the definition of the process. In order to prove that the process is not Markovian, we will use the following notation.
 	\begin{eqnarray*}
 		X_3 &=&\max\{U^{1/\alpha},V^{1/\beta} \}\\
 		X_2 &=&\max\{V^{1/\alpha},W^{1/\beta} \}\\ 
 		X_1 &=& \max\{W^{1/\alpha},Z^{1/\beta} \}
 	\end{eqnarray*}
 	where $U,V, W$ and $Z$ are i.i.d. $uniform(0,1)$ random variables. Two relevant conditional distribution functions are:
 	\begin{eqnarray}
 		P(X_3 \leq x_3 | X_2 \leq x_2) &=& P\Bigg( \max \{ U^{1/\alpha}, V^{1/\beta}\} \leq x_3 \Big| \max\{V^{1/\alpha},W^{1/\beta}\} \leq x_2  \Bigg)  \nonumber\\
 		&=& P\Bigg( U \leq x_3^{\alpha}, V \leq x^{\beta}_3   \Big| \max\{V^{1/\alpha},W^{1/\beta}\} \leq x_2  \Bigg)  \nonumber  \\
 		&=& P\Big( U \leq x_3^{\alpha}\Big) P\Bigg( V \leq x^{\beta}_3 \Big| \max\{V^{1/\alpha},W^{1/\beta}\} \leq x_2  \Bigg) 
 	\end{eqnarray}
 	since $U$ is independent of $V$ and $W $,  and 
 	\begin{eqnarray}
 		P(X_3 \leq x_3 | X_2 \leq x_2, X_1  \leq x_1) &=& P\Bigg( \max \{ U^{1/\alpha}, V^{1/\beta}\} \leq x_3 \Big| \max\{V^{1/\alpha},W^{1/\beta}\} \leq x_2,
 		\nonumber 	\\ &&  \max\{W^{1/\alpha},Z^{1/\beta}\} \leq x_1   \Bigg) \nonumber \\
 		&=&  P \Bigg( U\leq x^{\alpha}_3,  V\leq x^{\beta}_3  \Big| \max\{V^{1/\alpha},W^{1/\beta}\} \leq x_2, 
 		\nonumber 	\\ &&  \max\{W^{1/\alpha},Z^{1/\beta}\} \leq x_1   \Bigg) \nonumber \\
 		&=&  P\Big( U\leq x^{\alpha}_3\Big) P\Bigg(  V\leq x^{\beta}_3  \Big|  \max\{V^{1/\alpha},W^{1/\beta}\} \leq x_2, 
 		\nonumber 	\\ &&  \max\{W^{1/\alpha},Z^{1/\beta}\} \leq x_1   \Bigg)
 	\end{eqnarray}
 	since $U$ is independent of $V$,$W$ and $Z $. It is evident that the expression in $(2.5)$ is a function of $x_1$ and is not equal to the expression in $(2.4)$ which shows that the process does not have the Markov property. 
 \end{proof}
 
 \begin{corollary}
 	If the sequence of random variables $\{X_n\}$ is as defined in equation $(2.1)$, then the one-dimensional density of the process is 
 	\begin{eqnarray}
 		f(x_n) = (\alpha + \beta) x^{(\alpha + \beta)-1};   \mbox{      } 0\leq x_n \leq 1.
 	\end{eqnarray}
 \end{corollary}
 \begin{proof}
 	The proof will follow by taking differentiation of part $(a)$ of Theorem 1. 
 \end{proof}
 
 \begin{corollary}
 	If the sequence of random variables $\{X_n\}$ is as defined in equation $(2.1)$, then the mean and variance are 
 	\begin{eqnarray}
 		E(X_n) &=&  \frac{\alpha + \beta}{\alpha + \beta + 1}.  \\
 		V(X_n) &=&  \frac{\alpha + \beta}{(\alpha + \beta +1)^2(\alpha + \beta + 2)}.
 	\end{eqnarray}
 \end{corollary}
 \begin{proof}
 	These are trivial consequences of Corrollary 1.
 \end{proof}

 We refer to  Figure \ref{fig3} for the bivariate distribution function plots for different values of $\alpha$ and $\beta$.
 \begin{figure} \label{fig3} [h!]
 	\begin{subfigure}{8cm}
 		\centering
 		\includegraphics[width=10cm]{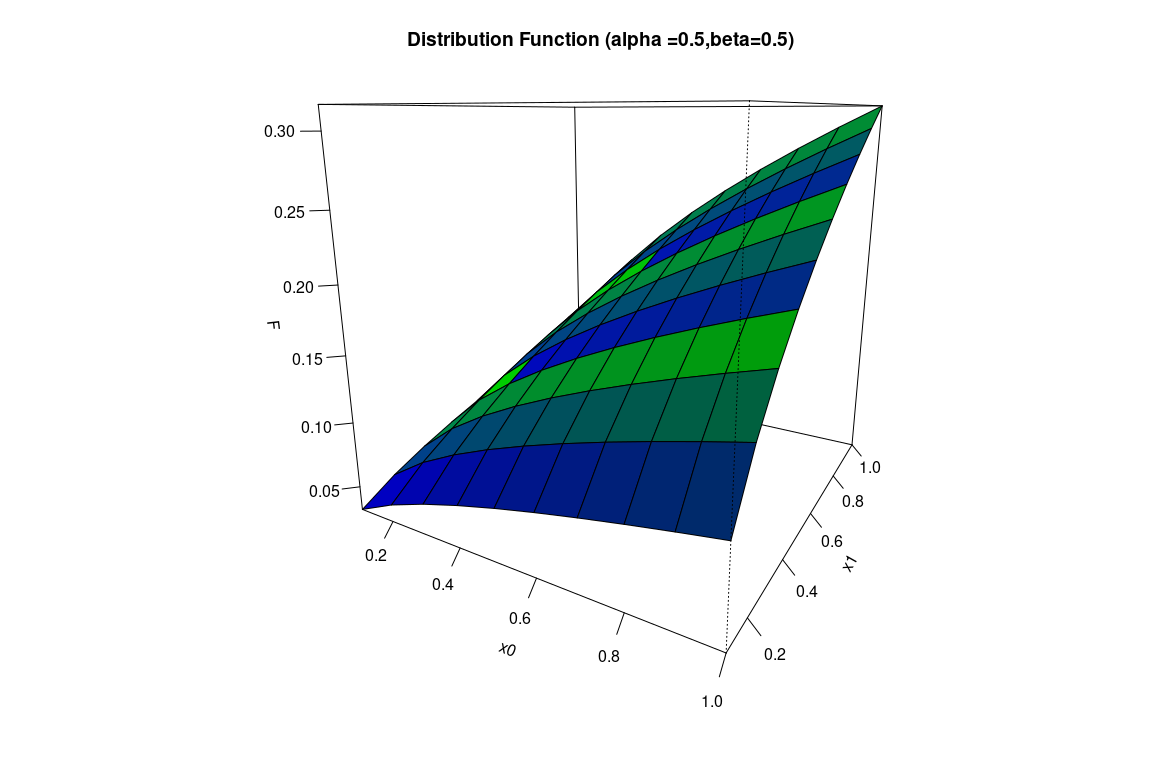}
 		\caption{$\alpha =0.5$ and $\beta =0.5$} 
 	\end{subfigure}
 	\begin{subfigure}{8cm}
 		\centering
 		\includegraphics[width=10cm]{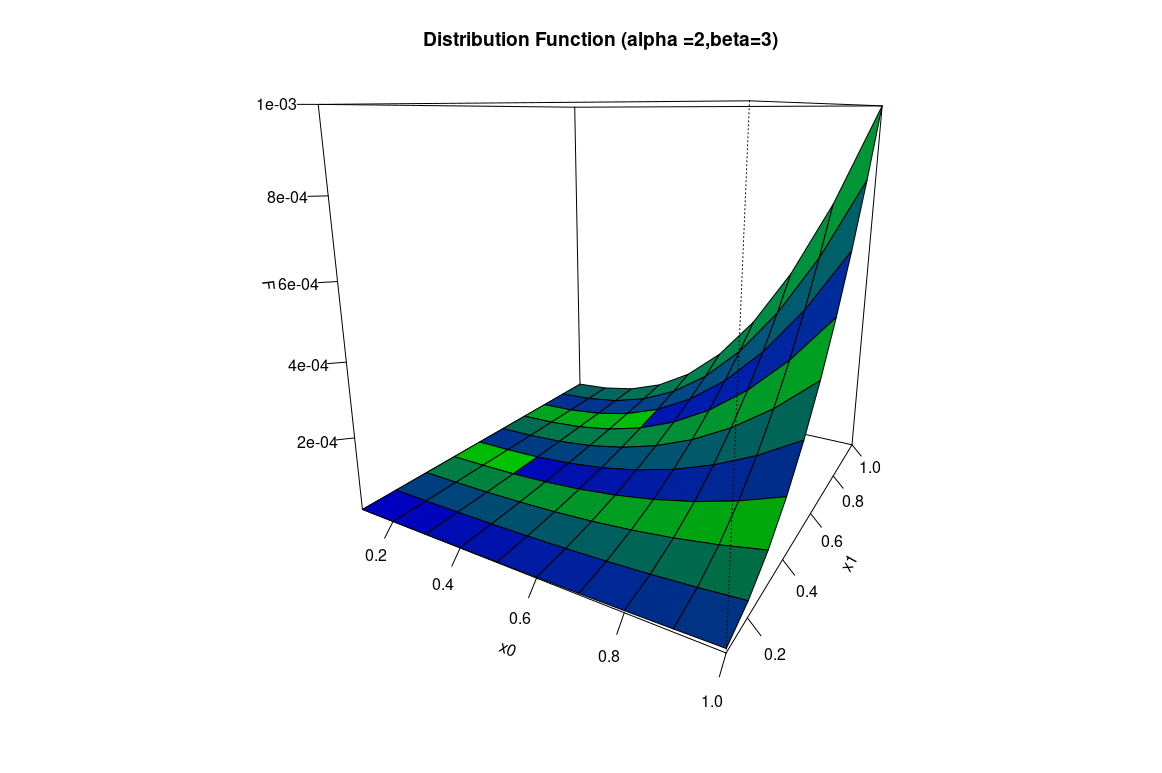}
 		\caption{$\alpha =2$ and $\beta =3$} 
 	\end{subfigure}
 	\caption{Distribution function}
 	\label{fig3}
 \end{figure}
 
 \begin{theorem}
 	If the sequence of random variables $\{X_n\}$ is as defined in equation (2.1), the auto-correlation function is 
 	\begin{eqnarray}
 		Corr(X_n,X_{n-1}) &=& \frac{Cov(X_n, X_{n-1})}{\sqrt{Var(X_n) Var(X_{n-1})}} \nonumber \\
 		&=& \frac{E(X_nX_{n-1})-\Big( \frac{\alpha +\beta}{\alpha + \beta +1 }\Big)^2}{\frac{\alpha +\beta}{(\alpha + \beta +1 )^2(\alpha+\beta +2)}}.
 	\end{eqnarray}
 \end{theorem}
 \begin{proof}
 	Indeed, it suffices to derive only $E(X_nX_{n-1})$ in order to compute auto-correlation function of the process.   For  $E(X_n X_{n-1})$ we argue as follows: 
 	\begin{eqnarray}
 		E(X_nX_{n-1})&=&E(V^{1/\beta}W^{1/\beta} I(U<V^{\alpha/\beta},W>V^{\beta/\alpha} ) ) \nonumber\\
 		& &+E(V^{1/\beta}V^{1/\alpha} I(U<V^{\alpha/\beta},W<V^{\beta/\alpha})) \nonumber \\
 		& &+E(U^{1/\alpha}W^{1/\beta} I(U>V^{\alpha/\beta},W>V^{\beta/\alpha}) )\nonumber \\
 		& &+E(U^{1/\alpha}V^{1/\beta} I(U>V^{\alpha/\beta},W<V^{\beta/\alpha}) ) \nonumber\\
 		& & =A+B+C+D.
 	\end{eqnarray}
 	
 	For term A, we have
 	\begin{eqnarray}
 		A &=&\int_0^1  dv\int_0^{v^{\alpha/\beta}} du\int_{v^{\beta/\alpha}}^1  dw \ [v^{1/\beta} w^{1/\beta} ]  \nonumber\\
 		&=& \int_0^1  dv\int_{v^{\beta/\alpha}}^1  dw \ [v^{\alpha/\beta}v^{1/\beta} w^{1/\beta} ] \nonumber  \\
 		&=& \frac{\beta^2}{\beta +1} \Bigg[ \frac{1}{\alpha + \beta +1} - \frac{\alpha}{\beta^2 + \alpha^2 +  \alpha \beta + \alpha +\beta}\Bigg].
 	\end{eqnarray}
 	
 	For term B, we have
 	\begin{eqnarray}
 		B &=&\int_0^1 dv\int_0^{v^{\alpha/\beta}} du \int^{v^{\beta/\alpha}}_0 dw \  [v^{1/\beta}  v^{1/\alpha}] \nonumber \\
 		&=&\int_0^1 dv \int^{v^{\beta/\alpha}}_0 dw \  [v^{\alpha/\beta}v^{1/\beta}  v^{1/\alpha}] \nonumber \\
 		&=&\int_0^1 dv  \  [v^{\beta/\alpha}v^{\alpha/\beta}v^{1/\beta}  v^{1/\alpha}] \nonumber \\
 		&=&\int_0^1v^{(\alpha+\beta+\alpha^2+\beta^2)/(\alpha\beta)} \ dv \nonumber \\
 		&=&\left[\frac{\alpha+\beta+\alpha^2+\beta^2+\alpha\beta}{\alpha\beta}\right]^{-1}.
 	\end{eqnarray}
 	
 	For term C, we have,
 	\begin{eqnarray*}
 		C &=&\int_0^1 dv\int^1_{v^{\alpha/\beta}} du\int_{v^{\beta/\alpha}}^1 dw \ [u^{1/\alpha} w^{1/\beta} ] \nonumber \\
 		&=& \frac{\beta \alpha}{(\beta +1)(\alpha +1)} \Bigg[ 1 - \frac{ \beta}{\alpha + \beta + 1 } - \frac{ \alpha}{\alpha + \beta + 1 }
 		+ \frac{\alpha \beta}{\alpha^2 + \beta^2 +  \alpha \beta + \alpha + \beta} \Bigg].
 	\end{eqnarray*}

 	For term D, we have
 	\begin{eqnarray}
 		D &=&\int_0^1 dv \int^1_{v^{\alpha/\beta}} du \  [v^{\beta/\alpha}u^{1/\alpha} v^{1/\beta}] \nonumber \\
 		&=& \frac{\alpha^2 \beta}{\alpha +1} \Bigg[ \frac{1}{\beta^2 + \alpha + \alpha \beta} - \frac{1}{\alpha^2 + \beta^2  + 2 \alpha + \alpha \beta}\Bigg]. 
 	\end{eqnarray}
 	Unfortunately, the resulting expression for $A+B+C+D$ is complicated and does not appear to simplify. 
 \end{proof}

 We refer to Figure \ref{fig4}  for the auto-correlation plot for different values of $\alpha$ and $\beta$.  The Figure \ref{fig5} refer to the auto-correlation plot for fixed value of $\alpha$ against varying $\beta$ and vice versa.  
 \begin{figure}[h!] 
 	\centering
 	\includegraphics[width=1.2\linewidth]{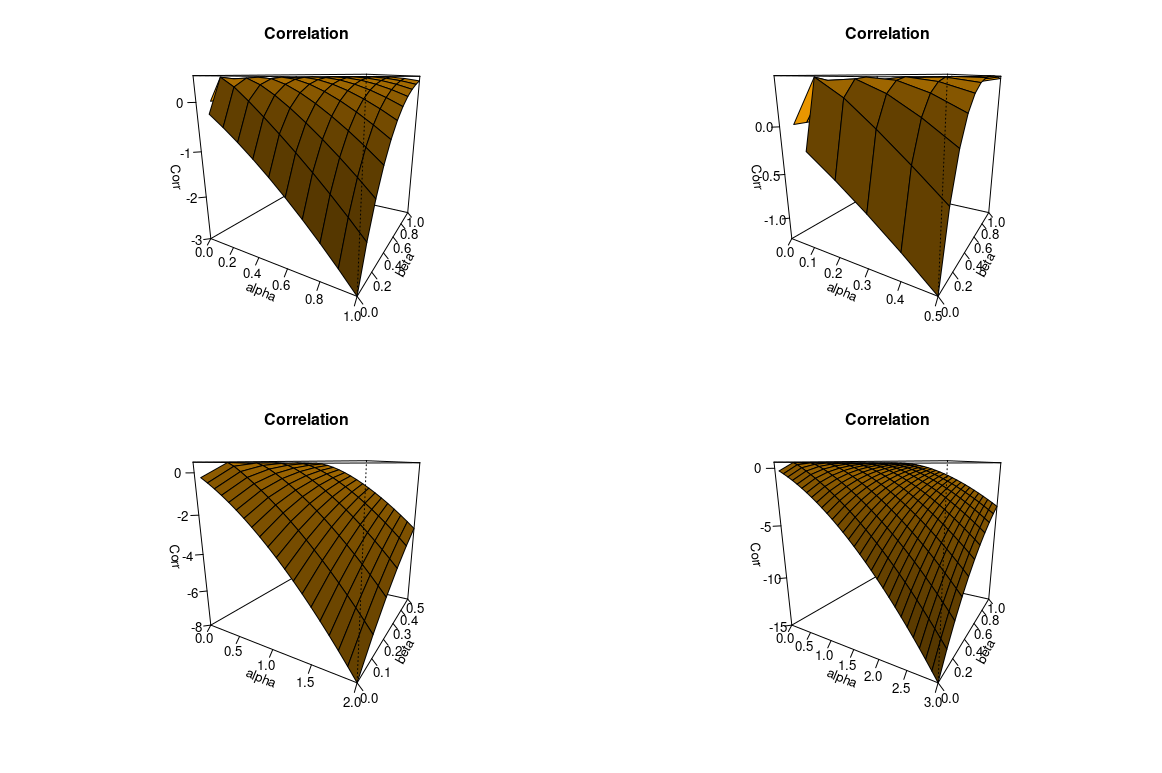} 
 	\caption{Correlation plots for $(\alpha,\beta) \in \{(1,1), (0.5,1),(2,5),(3,1)\}$ starting from $(0,0)$} 
 	\label{fig4} 
 \end{figure}

 \begin{figure} 
 	\begin{subfigure}{9cm}
 		\centering
 		\includegraphics[width=7cm]{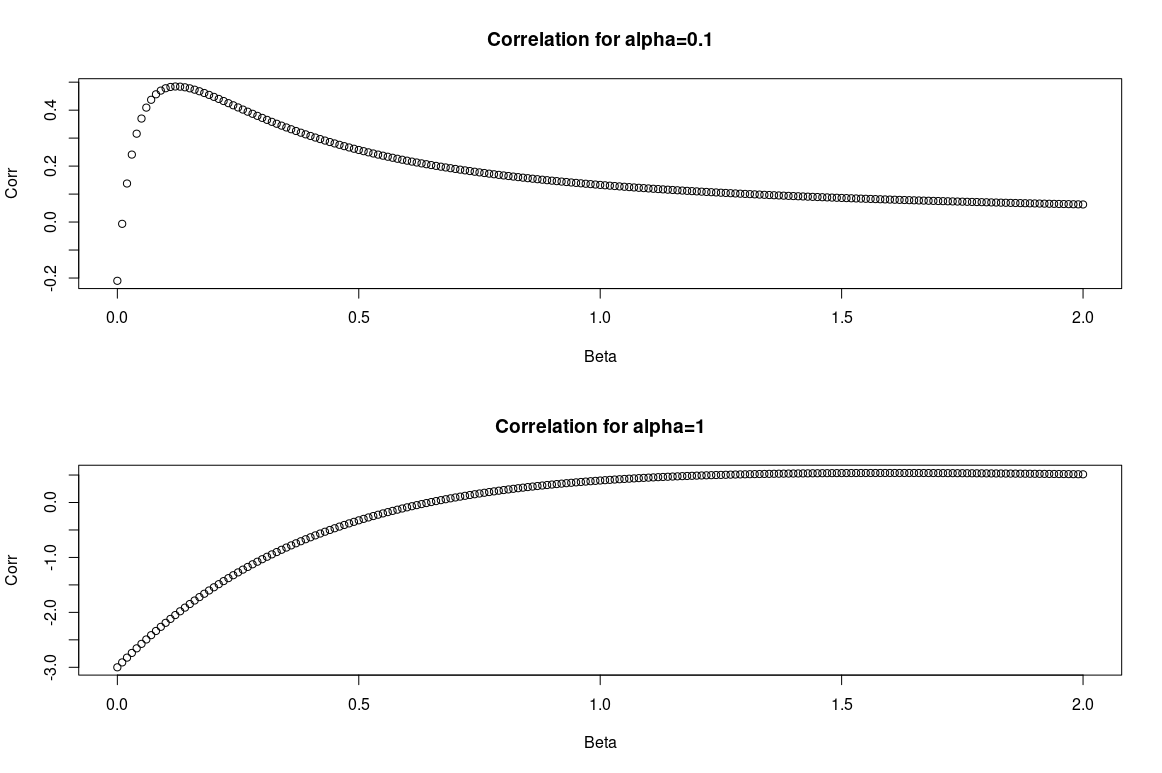}
 		\caption{Correlation plots for $\alpha =0.1,1$ and $\beta \in (0,2)$} 
 	\end{subfigure}
 	\begin{subfigure}{9cm}
 		\centering
 		\includegraphics[width=7cm]{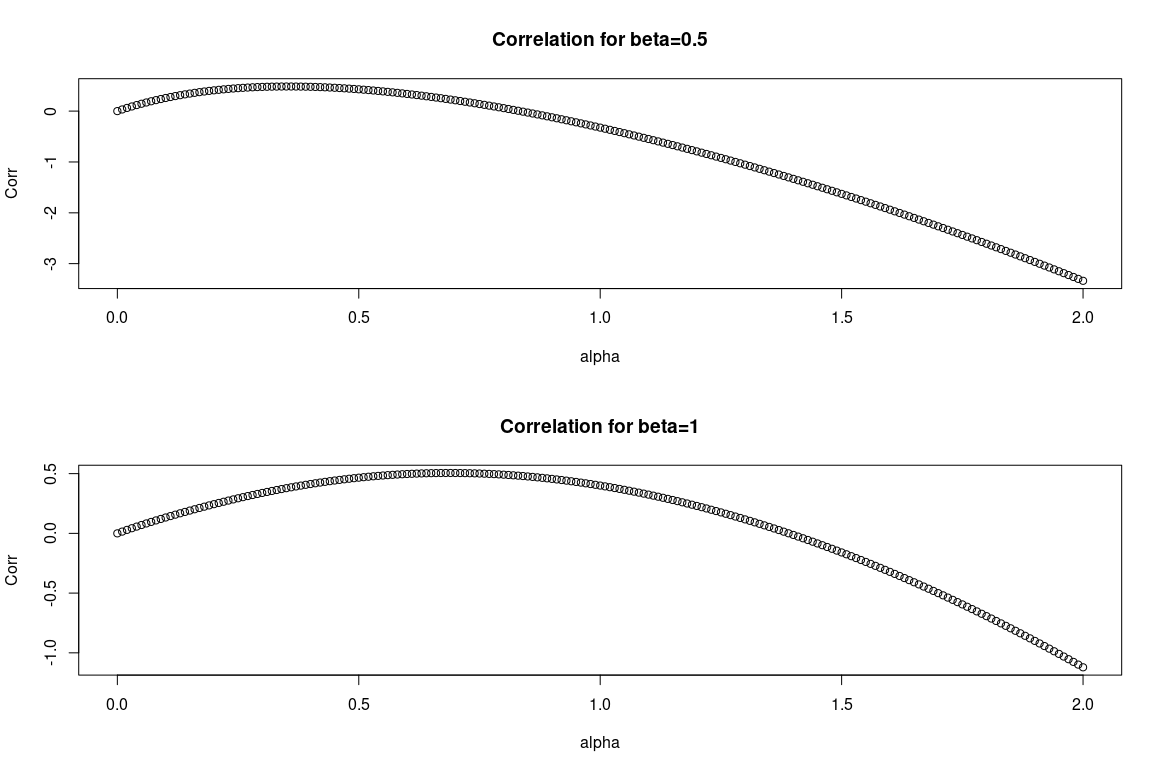}
 		\caption{Correlation plots for $\beta =0.1,1$ and $\alpha \in (0,2)$} 
 	\end{subfigure}
 	\caption{Correlation plots}
 	\label{fig5}
 \end{figure}

 \begin{theorem}
 	If the sequence of random variables $\{X_n\}$ is as defined in equation (2.1), then 
 	\begin{eqnarray}
 		P(X_1<X_2) = \left\{
 		\begin{array}{ll}
 			\frac{\alpha}{2\alpha+\beta} & \mbox{if } \alpha > \beta \\
 			\frac{\beta+\alpha}{2\beta+\alpha} & \mbox{if } \alpha \leq \beta. 
 		\end{array}
 		\right.
 	\end{eqnarray}
 \end{theorem}
 \begin{proof}
 	
 	By definition of the process  we have
 	$ X_1=\max \{U_1^{1/\alpha},U_0^{1/\beta}\}$ and $ X_2=\max \{U_2^{1/\alpha},U_1^{1/\beta}\}$
 	
 	from equation $(2.1)$,  where $U_0,U_1$ and $U_2$ are i.i,d $uniform(0,1)$ variables.  Then, 
 	\begin{eqnarray}
 		P(X_1<X_2)&=&P(\max\{U_1^{1/\alpha},U_0^{1/\beta}\}<\max\{U_2^{1/\alpha},U_1^{1/\beta}\}) \nonumber	\\
 		&=&\int_0^1P(\max\{u_1^{1/\alpha},U_0^{1/\beta}\}<\max\{U_2^{1/\alpha},u_1^{1/\beta}\}) \  du_1.
 	\end{eqnarray}
 	So if, $\alpha>\beta$  then 
 	\begin{eqnarray}	
 		P(X_1<X_2)	&=&\int_0^1 P(u_1^{1/\alpha}<U_2^{1/\alpha},U_0^{1/\beta}<U_2^{1/\alpha}) \ du_1 \nonumber
 		\\
 		&=&\int_0^1\int_0^1 P(u_1^{1/\alpha}<U_2^{1/\alpha},u_0^{1/\beta}<U_2^{1/\alpha}) \ du_1 \ du_0
 		\nonumber
 		\\
 		&=&\int_0^1\int_0^1 P(u_1<U_2,u_0^{\alpha/\beta}<U_2) \ du_1 \ du_0
 		\nonumber
 		\\&=&\int_0^1\int_0^1 (1-\max\{u_1,u_0^{\alpha/\beta}\}) \ du_1 \ du_0\nonumber
 		\\&=&1-\int_0^1\int_0^1\max\{u_1,u_0^{\alpha/\beta}\} \ du_1 \ du_0.
 		\nonumber \\
 	\end{eqnarray}
 	Thus we have
 	
 	\begin{eqnarray}
 		P(X_1<X_2)	&=&1-\int_0^1\int_0^{u_0^{\alpha/\beta}} u_0^{\alpha/\beta}  \ du_1 \ du_0  -\int_0^1\int_{u_0^{\alpha/\beta}}^1 u_1 \ du_1 \ du_0 \nonumber
 		\\ 
 		&=& 1-\int_0^1 u_0^{(2\alpha/\beta)} \ du_0 -(1/2)\int_0^1 [1-u_0^{(2\alpha/\beta)}] \ du_0 \nonumber
 		\\
 		&=& 1-\frac{1}{(2\alpha/\beta)+1}- (1/2)+(1/2)\int_0^1 u_0^{(2\alpha/\beta)} \ du_0 \nonumber
 		\\
 		&=& 1- \frac{\beta}{2 \alpha + \beta} -\frac{1}{2}  + \frac{1}{2} \frac{\beta}{2 \alpha + \beta}\nonumber
 		\\
 		&=& \frac{\alpha}{2\alpha+\beta}.
 	\end{eqnarray}

 	\bigskip
 	We can also consider the case in which $\alpha<\beta$.  By evaluating $P(X_1>X_2)$ when $\alpha<\beta$ we find by a parallel argument (just interchanging the roles of $\alpha$ and $\beta$)  that
 	$$ P(X_1>X_2)=\frac{\beta}{2\beta+\alpha}.$$
 	
 	Consequently, if $\alpha<\beta$ we have $P(X_1<X_2)=\frac{\beta+\alpha}{2\beta+\alpha}.$ 
 \end{proof}
 \bigskip
 
 Finally, we make a note that the general PRH process is defined in terms of a general distribution function $F_0(x)$ with quantile function $F_0^{-1}(u)$ and support $(0,\infty)$. Denoting the PRH process by $ \left\{ Y_n \right\} $  we then have for each $n$:
 \begin{eqnarray} \label{PFD-PRH}
 	Y_n=F_0^{-1}(X_n)\\
 	X_n=F_0(Y_n).
 \end{eqnarray}
 
 If the index set of the process, the set of values of $n$ for which it is defined, consists of the positive and negative integers then the processes will be stationary. If the index set is $0,1,2..$ then it will be necessary to define $U_{-1}$ to have a uniform distribution in order for $X_0$ to be well-defined.  With that modifation the process will be stationary, despite the claim in Kundu \cite{kd22}, the processes will not be Markovian (see above part $(c)$ of Theorem 1).
 
 \bigskip 
 \begin{remark}
 	
 	Note that if $X\sim PFD(\alpha)$ then $Y=1/X \sim Pareto(\alpha,1)$. So we can recast all the discussion in this paper in terms of Pareto processes.
 \end{remark}
 
 \section{Notes on other examples}
 A key property of power function distributions is the following.  If $V_i, \ \ i=1,2$ are independent with
 $V_ i\sim PFD(\alpha_i), \ \  i=1,2$ then $W= \max\{V_1,V_2\} \sim PFD(\alpha_1+\alpha_2).$ This property of the PF distribution was crucial in the construction of the Kundu PFD process.
 
 \bigskip
 
 We may make use of this observation to define a stationary power function$(\alpha)$ process as follows:
 Define $X_0=U_0^{1/\alpha} \sim PF(\alpha)$ and for $n=1,2,...$ define
 
 \begin{equation}\label{ARPFProc}
 	X_n= \max \{X_{n-1}^{\alpha/(\alpha-\delta)},U_n^{1/\delta}\},
 \end{equation}
 where $\delta \in (0,\alpha)$ and the $U_n$'s are i.i.d. $uniform(0,1)$ random variables.  It is not difficult to verify that this process is completely stationary with $X_n \stackrel{d}{=}U^{1/\alpha} \sim PFD(\alpha)$ for every $n$.   It is clearly a Markov process. Even if $X_0$ does not have a $PF(\alpha)$ distribution, the process is still Markovian though no longer stationary. It will, however, have a $PF(\alpha)$ long run distribuition.  We refer to Figures \ref{fig6}--\ref{fig7} for examples of simulated sample paths of steps $n=10$ and $n=100$.
 
 \begin{figure} [h!]
 	\begin{subfigure}{9cm}
 		\centering
 		\includegraphics[width=7cm]{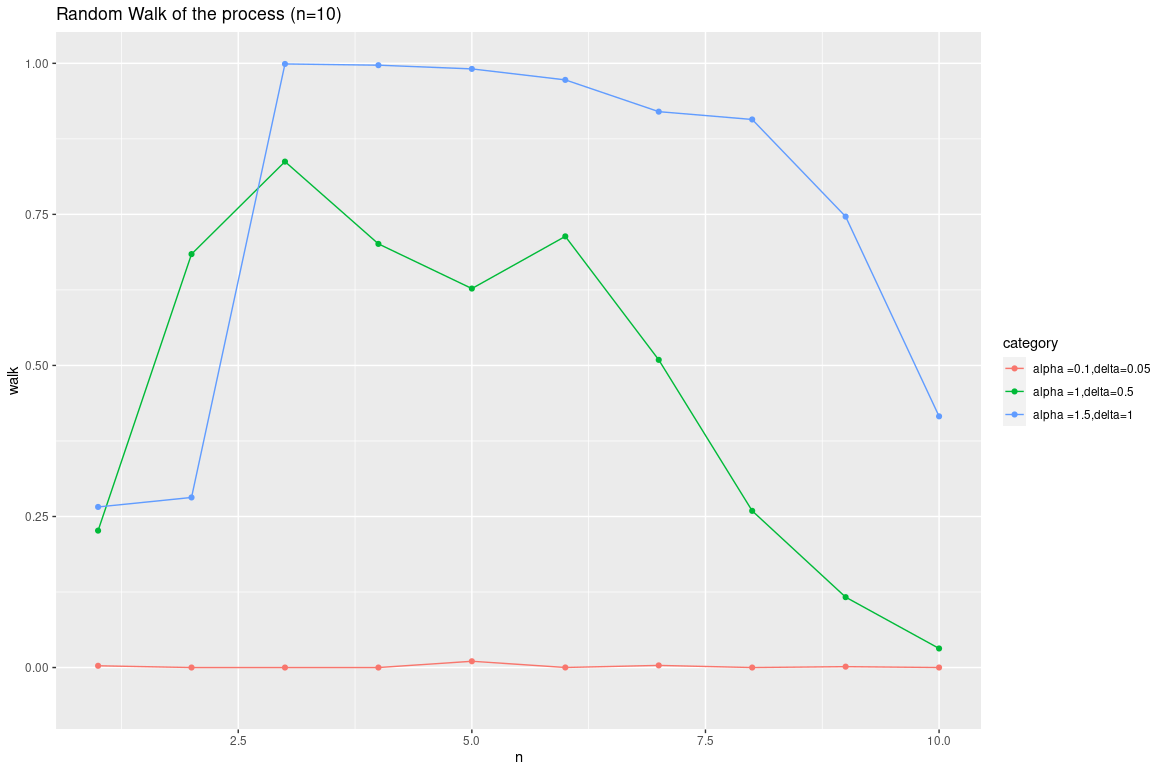}
 		\caption{$n=10$}  
 	\end{subfigure}
 	\begin{subfigure}{9cm}
 		\centering
 		\includegraphics[width=7cm]{walk101.png}
 		\caption{$n=10$} 
 	\end{subfigure}
 	\caption{Sample Path with $10$ steps}
 	\label{fig6}
 \end{figure}

 \begin{figure} 
 	\begin{subfigure}{9cm}
 		\centering
 		\includegraphics[width=7cm]{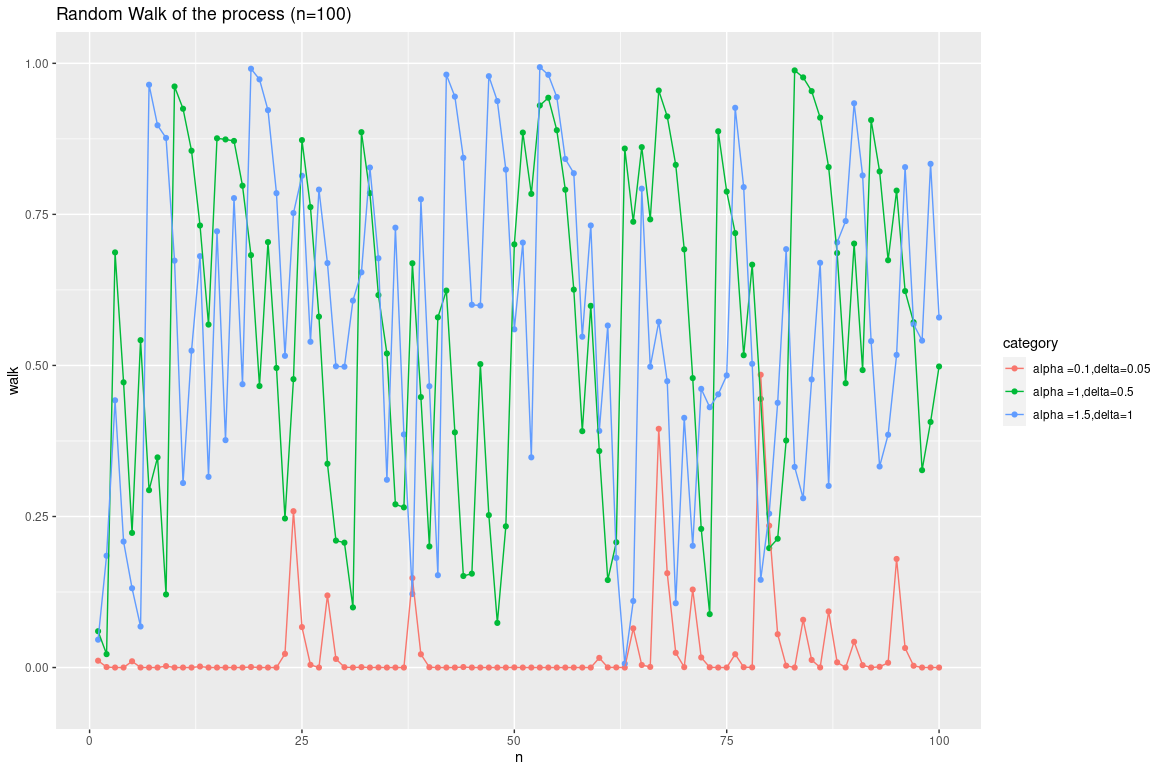}
 		\caption{$n=100$}  
 	\end{subfigure}
 	\begin{subfigure}{9cm}
 		\centering
 		\includegraphics[width=7cm]{walk1001.png}
 		\caption{$n=100$} 
 	\end{subfigure}
 	\caption{Sample Path with $100$ steps}
 	\label{fig7}
 \end{figure}
 
 Processes of the kind $X_n = c \Big( \min \{X_{n-1}, Y_n \} \Big)$, where $c > 1$ be a positive constant,  $X_n$ and $Y_n$ are random variables having some  specific distributions are called minification processes and appear frequently in the literature, for examples see Arnold and Hallet \cite{ah89} and Arnold \cite{a01} and the references in those papers.
 
 In the following subsection we derive some stochastic properties of the process defined in $(3.1)$.    
 \subsection{Stochastic properties of the power function process defined in $(3.1)$}   
 
 \begin{theorem}
 	If the sequence of random variables $\{ X_n\}$ is as defined in equation $(3.1)$, the following statements are true:
 	\begin{enumerate} [label=(\alph*)]
 		\item $\{X_n\}$ is stationary Markov process. 
 		\item $\{X_n\}$ has a PFD($\alpha$) distribution for each $n$..
 		\item the joint distribution function of $(X_{n-1},X_n)$  is $x_1^\delta \min \{x_0^\alpha, x_1^{\alpha-\delta}\}, \ \  0 \leq x_0,x_1 \leq 1$.
 	\end{enumerate}
 \end{theorem}      
 
 \begin{proof}
 	Part $(a)$ is trivial. For part $(b)$ of the theorem, we prove it using mathematical induction. Now, consider
 	the distribution function of $\{ X_n\}$ is  
 	\begin{eqnarray}
 		F_n(x_n)=P(X_n \leq x) &=& P\Big( \max\Bigl\{ X_{n-1}^{\frac{\alpha}{\alpha-\delta}}, U^{\frac{1}{\delta}}_n \Bigr\} \leq x \Big) \nonumber \\
 		&=& P\Big( X_{n-1}^{\frac{\alpha}{\alpha-\delta}} \leq x, U^{\frac{1}{\delta}}_n \leq x \Big) \nonumber \\
 		&=& P\Big( X_{n-1}^{\frac{\alpha}{\alpha-\delta}} \leq x \Big)  P\Big( U^{\frac{1}{\delta}}_n \leq x \Big) \nonumber \\
 		&=& P \Big(X_{n-1} \leq x^{\frac{\alpha -\delta}{\alpha}} \Big)  P\Big(U_n \leq x^\delta\Big) \nonumber \\
 		&=& x^\delta P \Big(X_{n-1} \leq x^{\frac{\alpha -\delta}{\alpha}} \Big)
 	\end{eqnarray} 
 	
 	Now, for $n=1$ the distribution function will be
 	\begin{eqnarray}
 		F_1(x) = P(X_1 \leq x)  = x^\delta P\Big(X_0 \leq x^{\frac{\alpha -\delta}{\alpha}} \Big) = x^\delta P\Big( U^{\frac{1}{\alpha}}_0 \leq x^{\frac{\alpha -\delta}{\alpha}}\Big) = x^{\alpha}.
 	\end{eqnarray}
 	For $n=2$
 	\begin{eqnarray}
 		F_2(x) &=& x^\delta P \Big(X_{1} \leq x^{\frac{\alpha -\delta}{\alpha}} \Big) = x^\delta P\Big( \max\Bigl\{ X_{0}^{\frac{\alpha}{\alpha-\delta}}, U^{\frac{1}{\delta}}_1 \Bigr\} \leq x^{\frac{\alpha -\delta}{\alpha}} \Big) \nonumber \\
 		&=&  x^\delta P \Big( U^{\frac{1}{\alpha -\delta}}_0 \leq x^{\frac{\alpha -\delta}{\alpha}} \Big) P \Big( U^{\frac{1}{\delta}}_1 \leq x^{\frac{\alpha -\delta}{\alpha}} \Big) \nonumber \\
 		&=&  x^{\alpha + \frac{(\alpha -\delta)^2}{\alpha} + \frac{\delta(\alpha - \delta)}{\alpha}} \nonumber \\
 		&=& x^{\alpha}.
 	\end{eqnarray}
 	
 	By mathematical indiuction at $n=m$th step the distribution function of $X_m$ is
 	
 	\begin{eqnarray}
 		F_m(x)= P(X_m \leq x) &=&  x^{\alpha}.
 	\end{eqnarray}
 	Hence, for any $n$ the marginal distribution function of $\{ X_n\}$ follows PFP($\alpha$) is 
 	\begin{eqnarray}
 		F_n(x)= P(X_n \leq x) &=&  x^{\alpha}.
 	\end{eqnarray}

 	Now, for part $(c)$ of the theorem, the joint distribution of $X_n$ and $X_{n+1}$ is, by stationarity of the process,  the same as the joint distribution of $X_0$ and $X_1$ which may be computed as follows
 	\begin{eqnarray}
 		F_{X_0,X_1}(x_0,x_1) &=& P(X_0 \leq x_0, X_{1} \leq x_1)  \nonumber \\
 		&=& P(U_0^{1/\alpha}\leq x_0,U_0^{1/(\alpha-\delta)}\leq x_1,U_1^{1/\delta}\leq x_1) \nonumber \\
 		&=& P(U_0 \leq x_0^\alpha,U_0 \leq x_1^{\alpha-\delta},U_1 \leq x_1^\delta) \nonumber  \\
 		&=&P(U_0 \leq \min\{x_0^\alpha, x_1^{\alpha-\delta}\},U_1 \leq x_1^\delta)\}  \nonumber \\
 		&=&x_1^\delta \min \{x_0^\alpha, x_1^{\alpha-\delta}\}, \ \  0 \leq x_0,x_1 \leq 1.
 	\end{eqnarray}
 \end{proof}

 \begin{corollary} 
 	If the sequence of random variables $\{X_n\}$ is as defined in equation $(3.1)$, then the one-dimensional density of the process is
 	\begin{eqnarray}
 		f_n(x_n) = \alpha x_n^{\alpha-1};         \mbox{     } 0 <x_n <1.
 	\end{eqnarray}
 \end{corollary}
 \begin{proof}
 	The proof will follow by taking differentiation of part $(a)$ of Theorem 1.
 \end{proof}

 \begin{corollary} 
 	If the sequence of random variables $\{X_n\}$ is as defined in equation $(3.1)$, then the mean and variance are 
 	\begin{eqnarray}
 		E(X_n) &=& \frac{\alpha}{ \alpha +1} \\
 		Var(X_n) &=& \frac{\alpha}{(\alpha +1)^2(\alpha +2)}.
 	\end{eqnarray}
 \end{corollary}
 \begin{proof}
 	The proof is trivial from Corrollary 1. 
 \end{proof}
 
 We refer to  Figure \ref{fig8} for the distribution function plots for different values of $\alpha$ and $\delta$. Since this Markov process has as its initial distribution the stationary distribution, the process is completely stationary, with this assumption we derive the auto-correlation function. 
 
 \begin{figure} [h!]
 	\begin{subfigure}{8cm}
 		\centering
 		\includegraphics[width=10cm]{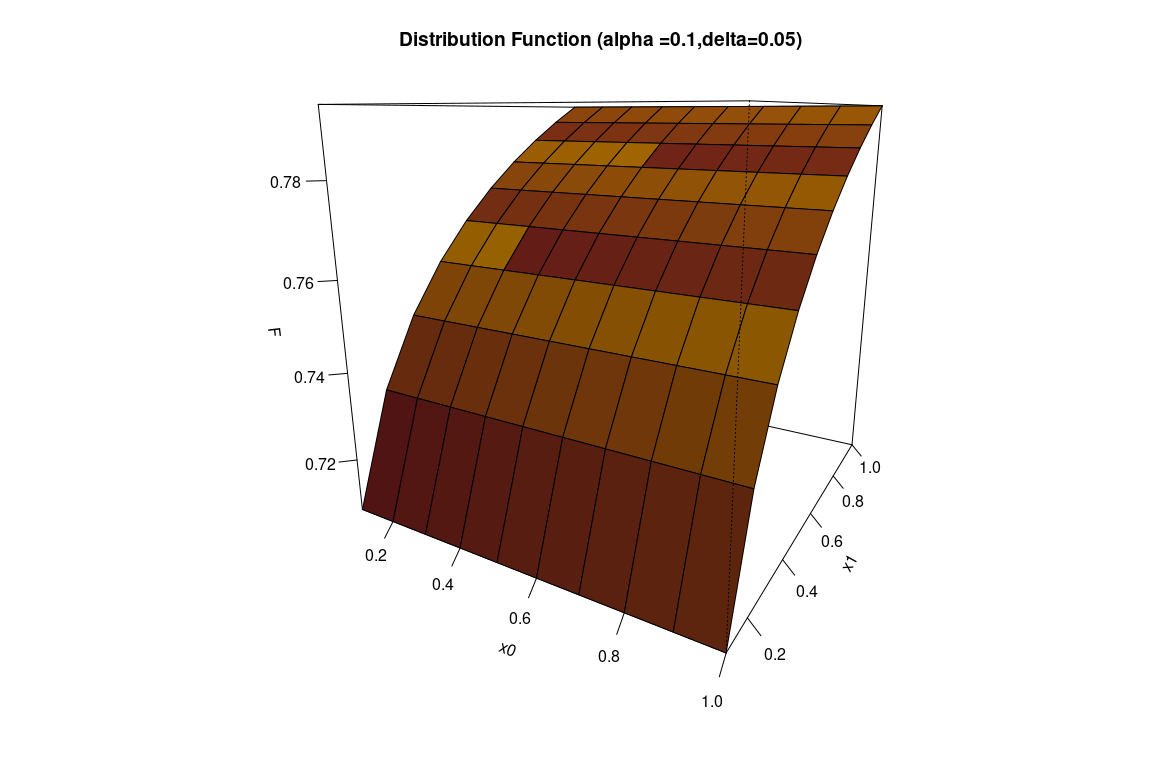}
 		\caption{$\alpha =0.1$ and $\delta =0.05$} 
 	\end{subfigure}
 	\begin{subfigure}{8cm}
 		\centering
 		\includegraphics[width=10cm]{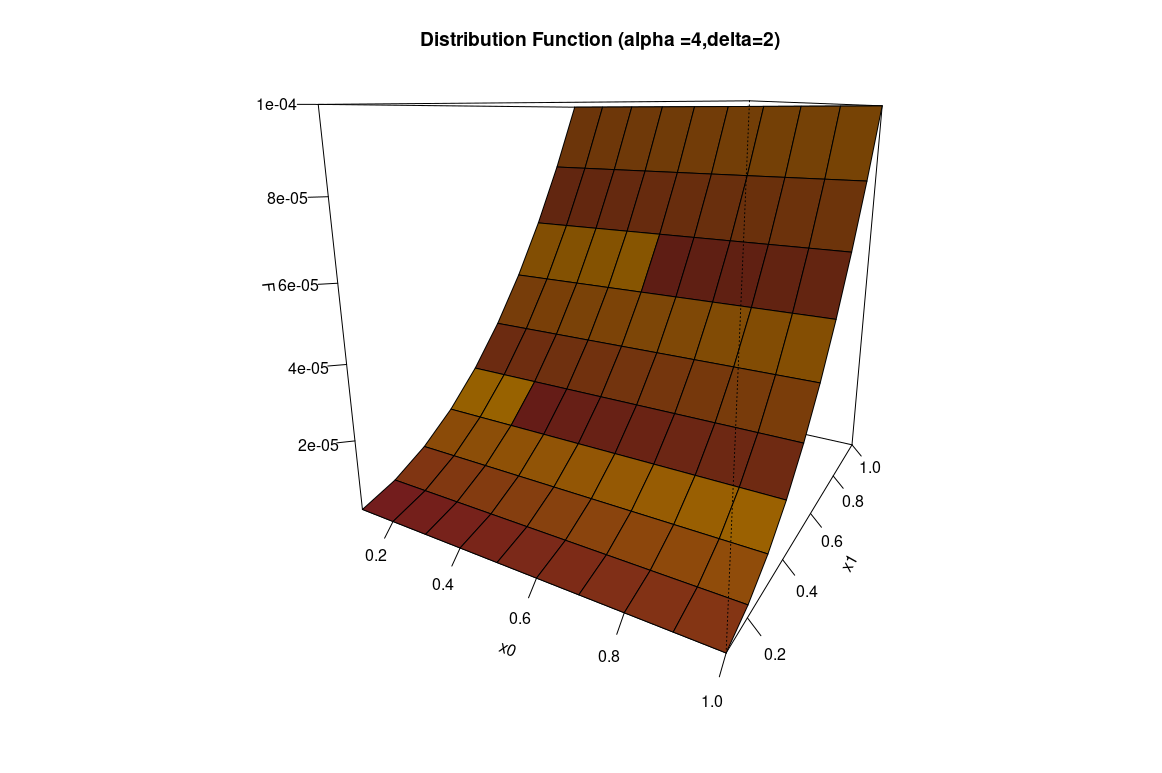}
 		\caption{$\alpha = 4$ and $\delta =2$} 
 	\end{subfigure}
 	\caption{Distribution function plots}
 	\label{fig8}
 \end{figure}
 
 \begin{theorem}
 	If the sequence of random variables $\{X_n\}$ is as defined in equation $(3.1)$, then the auto-correlation function will be
 	\begin{eqnarray}
 		Corr(X_0,X_n) = \frac{ \frac{\delta}{\delta +1} \Bigg[ \frac{\alpha}{\alpha +1} - \Bigg(\frac{1}{\alpha}+\frac{\delta+1}{\alpha-\delta} +1 \Bigg)^{-1}\Bigg] +\Bigg[\frac{1}{\alpha} + \frac{1+ \delta}{\alpha -\delta} +1\Bigg]^{-1} - \Bigg[\frac{\alpha}{ \alpha +1}\Bigg]^2}{\frac{\alpha}{(\alpha +1)^2(\alpha +2)}}	\nonumber\\
 	\end{eqnarray}
 	which is monotone in $\delta$.
 \end{theorem}
 \begin{proof} For  $E(X_0X_1)$ we have
 	\begin{eqnarray}
 		E(X_0X_1) &=& E\Big(U^{\frac{1}{\alpha}}_0 \max\Big \{ X^{\frac{\alpha}{\alpha -\delta}}_0, U^{1/\delta}_1\Big \} \Big) \nonumber \\
 		&=& E \Big( U^{\frac{1}{\alpha}}_0 \max\Big \{ U^{\frac{1}{\alpha -\delta}}_0, U^{1/\delta}_1\Big \}\Big) \nonumber\\
 		&=& \int_{0}^{1} u^{\frac{1}{\alpha}}_0 \Bigg[ \int_{u^{\frac{\delta}{\alpha-\delta}}_0}^{1}  u^{\frac{1}{\delta}}_1 du_1\Bigg] du_0  + \int_{0}^{1} u^{\frac{1}{\alpha}}_0 \Bigg[ \int_{0}^{u^{\frac{\delta}{\alpha-\delta}}_0}  u^{\frac{1}{\alpha - \delta}}_0 du_1\Bigg] du_0 \nonumber  \\
 		&=&  \frac{\delta}{\delta +1} \Bigg[ \frac{\alpha}{\alpha +1} - \Bigg(\frac{1}{\alpha}+\frac{\delta+1}{\alpha-\delta} +1 \Bigg)^{-1}\Bigg] + \Bigg[\frac{1}{\alpha} + \frac{1+ \delta}{\alpha -\delta} +1\Bigg]^{-1} 
 	\end{eqnarray}
 	
 	Thus, $Cov(X_0,X_1)$ is
 	\begin{eqnarray}
 		Cov(X_0,X_1)  = E(X_0 X_1) - E(X_0) E(X_1) &=&  \frac{\delta}{\delta +1} \Bigg[ \frac{\alpha}{\alpha +1} - \Bigg(\frac{1}{\alpha}+\frac{\delta+1}{\alpha-\delta} +1 \Bigg)^{-1}\Bigg] \nonumber \\ &&+ \Bigg[\frac{1}{\alpha} + \frac{1+ \delta}{\alpha -\delta} +1\Bigg]^{-1} - \Bigg[\frac{\alpha}{ \alpha +1}\Bigg]^2,
 	\end{eqnarray}
 	from which the correlatgion is of the form stated.
 \end{proof}
 
 \begin{remark}
 	From Figure \ref{fig9} it is evident that the auto-correlation function for the process in $(3.1)$ is a monotonic function in $\delta$.
 \end{remark}

 \begin{figure}[h!] 
 	\centering
 	\includegraphics[width=.7\linewidth]{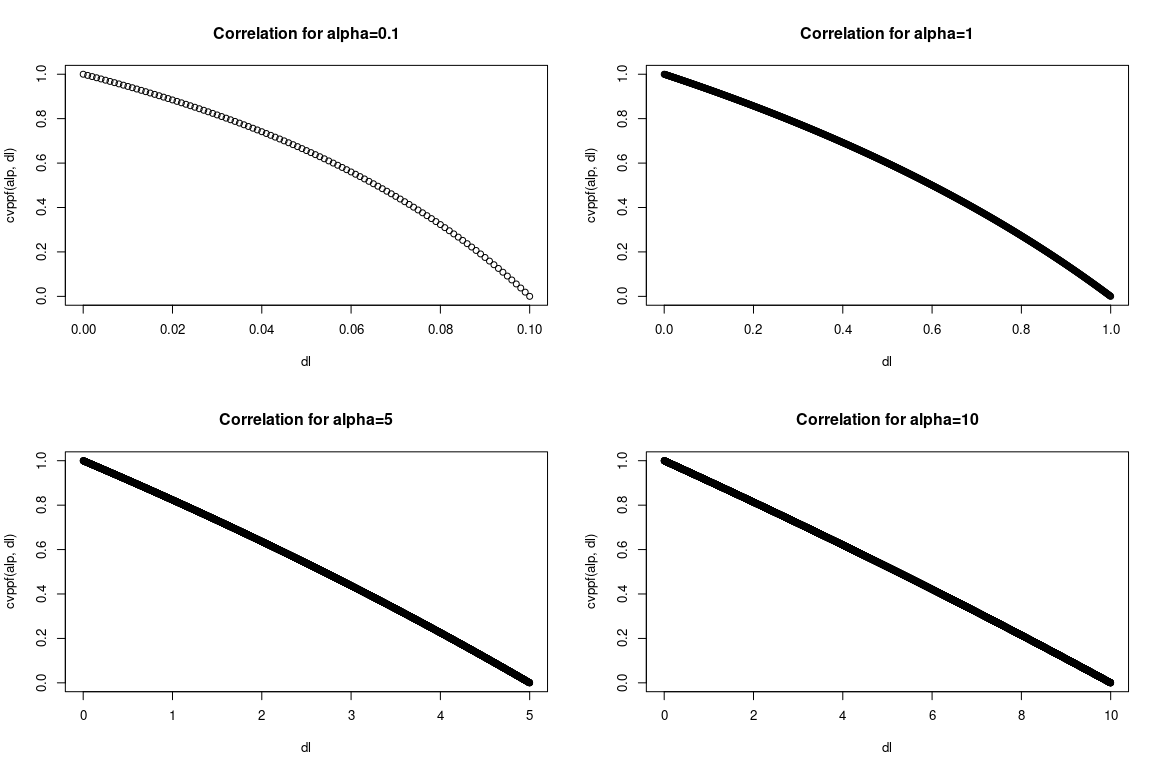}
 	\caption{Correlation plot for different values of $\alpha =0.1,1,5,10$ and $\delta$ ranging from $0$ to $\alpha$ } 
 	\label{fig9}
 \end{figure}

 \begin{remark}
 	The bivarite distribution function defined in $(3.7)$ does not have a density. Hence, likelihood based inferences such as maximum likelihood and the Bayesian approach cannot be used for estimating the parameters of the process.
 \end{remark}
 
 \begin{theorem}
 	If the sequence of random variables $\{ X_n\}$ is as defined in equation $(3.1)$ then 
 	\begin{eqnarray}
 		P(X_1<X_0) =\frac{\alpha}{\alpha+\delta}.
 	\end{eqnarray}
 \end{theorem}
 
 \begin{proof}
 	By definition of the process we have
 	$ X_1=\max\{X_0^{\alpha/(\alpha-\delta)},U_1^{1/\delta}\}$
 	from equation (3.1) with $X_0=U_0^{1/\alpha}$, where $U_0$ and $U_1$ are i.i,d $uniform(0,1)$ variables. 
 	We refer to Figures \ref{2rfig6}--\ref{2rfig7} for examples of simulated sample paths of steps $n=10$ and $n=100$ for the second order Kundu process.
 	Then, $P(X_1 <X_0)$ is given by 
 	\begin{eqnarray}
 		P(X_1<X_0)&=&P(\max\{X_0^{\alpha/(\alpha-\delta)},U_1^{1/\delta}\}<X_0) \nonumber
 		\\ &=&P(\max\{U_0^{1/(\alpha-\delta)},U_1^{1/\delta}\}<U_0^{1/\alpha}) \nonumber
 		\\&=&\int_0^1\left[P(\max\{u_0^{1/(\alpha-\delta)},U_1^{1/\delta}\}<u_0^{1/\alpha})\right]du_0 \nonumber
 		\\&=&\int_0^1\left[P(u_0^{1/(\alpha-\delta)}<u_0^{1/\alpha}, U_1^{1/\delta}<u_0^{1/\alpha})\right] du_0 \nonumber
 		\\  &=&\int_0^1\left[P( U_1^{1/\delta}<u_0^{1/\alpha})\right] du_0  \ \ \mbox{because} \ \ u_0^{1/(\alpha-\delta)}<u_0^{1/\alpha} \ \ \mbox{is always true} \nonumber
 		\\ &=&\int_0^1\left[P( U_1<u_0^{\delta/\alpha})\right] du_0 \nonumber
 		\\ &=&\frac{1}{(\delta/\alpha)+1} \nonumber
 		\\ &=&\frac{\alpha}{\alpha+\delta}.
 	\end{eqnarray}
 \end{proof}
 
 \section{Second (and higher) order processes}
 A second order Kundu process is readily defined by assuming that $U_{-2}.U_{-1},U_0,U_1,U_2,...$ are i.id. $uniform(0,1)$ random variables and for $n=0,1,2,...$ defining
 \begin{equation} \label{PFDP(2)}
 	X_n= \max \left\{U_n^{1/\alpha_0},U_{n-1}^{1/\alpha_1},U_{n-2}^{1/\alpha_2} \right\},
 \end{equation}
 where $\alpha_0, \alpha_1, \alpha_2 >0$.  This is a stationary process. It is readily verified that, for each $n$, 
 \begin{equation} 
 	X_n \sim PFD(\alpha_0+\alpha_1+\alpha_2).
 \end{equation}
 
 \begin{proposition}
 	If the sequence of random variables $\{X_n\}$ is as defined in equation $(4.1)$, then the one-dimensional density of the process is 
 	\begin{eqnarray}
 		f(x_n) = (\alpha_0 + \alpha_1 + \alpha_2) x^{(\alpha_0 + \alpha_1 + \alpha_2)-1};   \mbox{      } 0\leq x_n \leq 1.
 	\end{eqnarray}
 \end{proposition}
 
 \begin{proposition}
 	If the sequence of random variables $\{X_n\}$ is as defined in equation $(4.1)$, then the mean and variance are 
 	\begin{eqnarray}
 		E(X_n) &=&  \frac{\alpha_0 + \alpha_1 + \alpha_2}{\alpha_0 + \alpha_1 + \alpha_2 + 1}  \\
 		V(X_n) &=&  \frac{\alpha_0 + \alpha_1 + \alpha_2}{(\alpha_0 + \alpha_1 + \alpha_2 +1)^2(\alpha_0 + \alpha_1 + \alpha_2 + 2)}.
 	\end{eqnarray}
 \end{proposition}
 \begin{proof}
 	These are trivial consequences of Proposition 1. 
 \end{proof}
 
 We refer to Figures \ref{2rfig6}--\ref{2rfig7} for examples of simulated sample paths of steps $n=10$ and $n=100$. 
 
 \begin{figure} [h!]
 	\begin{subfigure}{9cm}
 		\centering
 		\includegraphics[width=7cm]{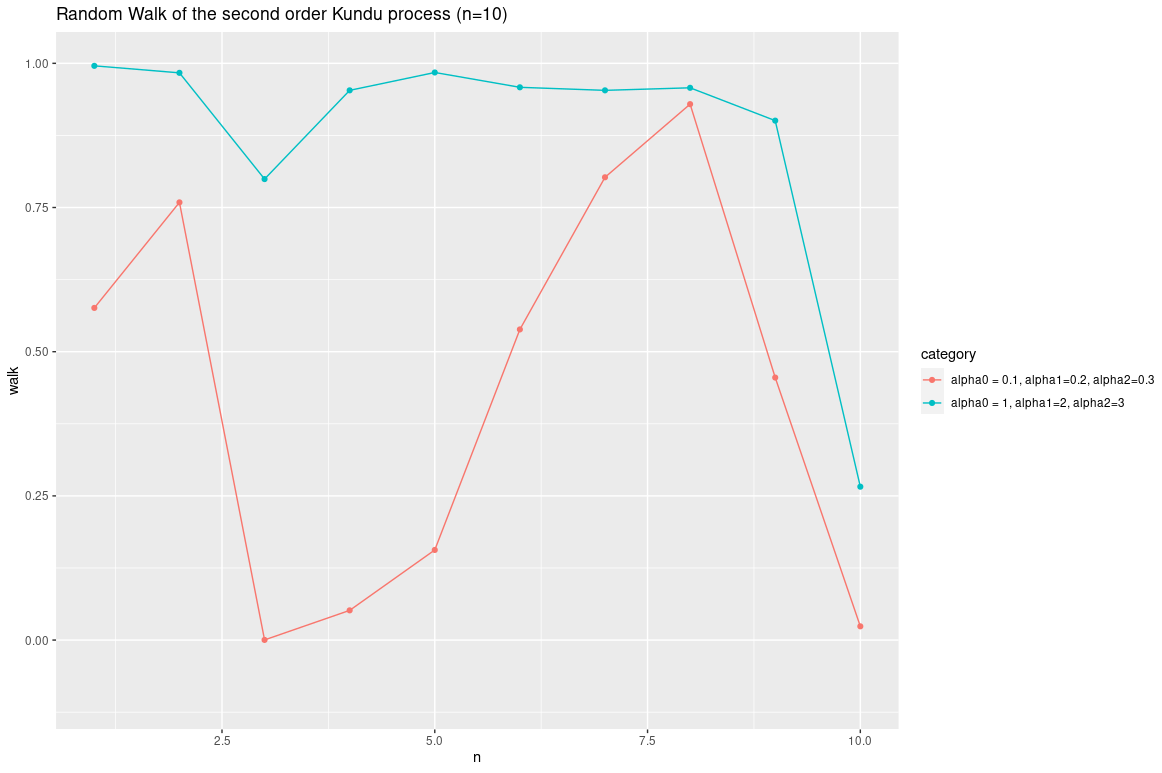}
 		\caption{$n=10$}  
 	\end{subfigure}
 	\begin{subfigure}{9cm}
 		\centering
 		\includegraphics[width=7cm]{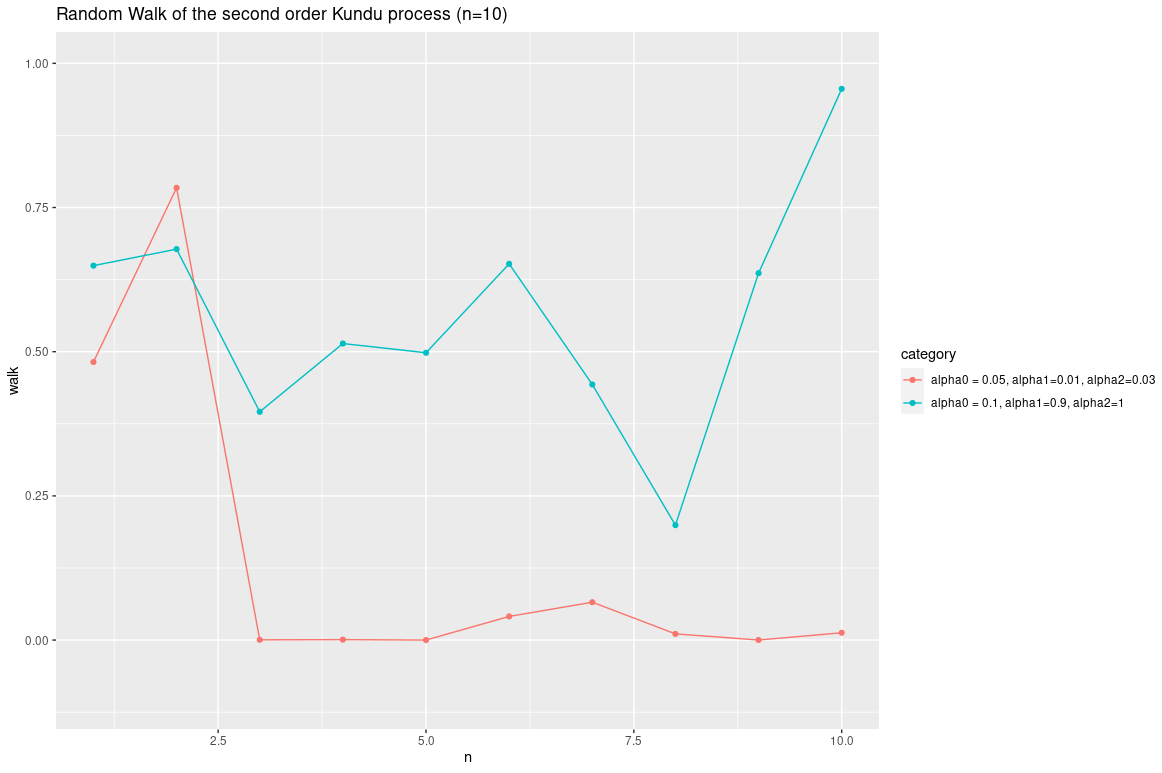}
 		\caption{$n=10$} 
 	\end{subfigure}
 	\caption{Sample Path with $10$ steps}
 	\label{2rfig6}
 \end{figure}
 
 \begin{figure} 
 	\begin{subfigure}{9cm}
 		\centering
 		\includegraphics[width=7cm]{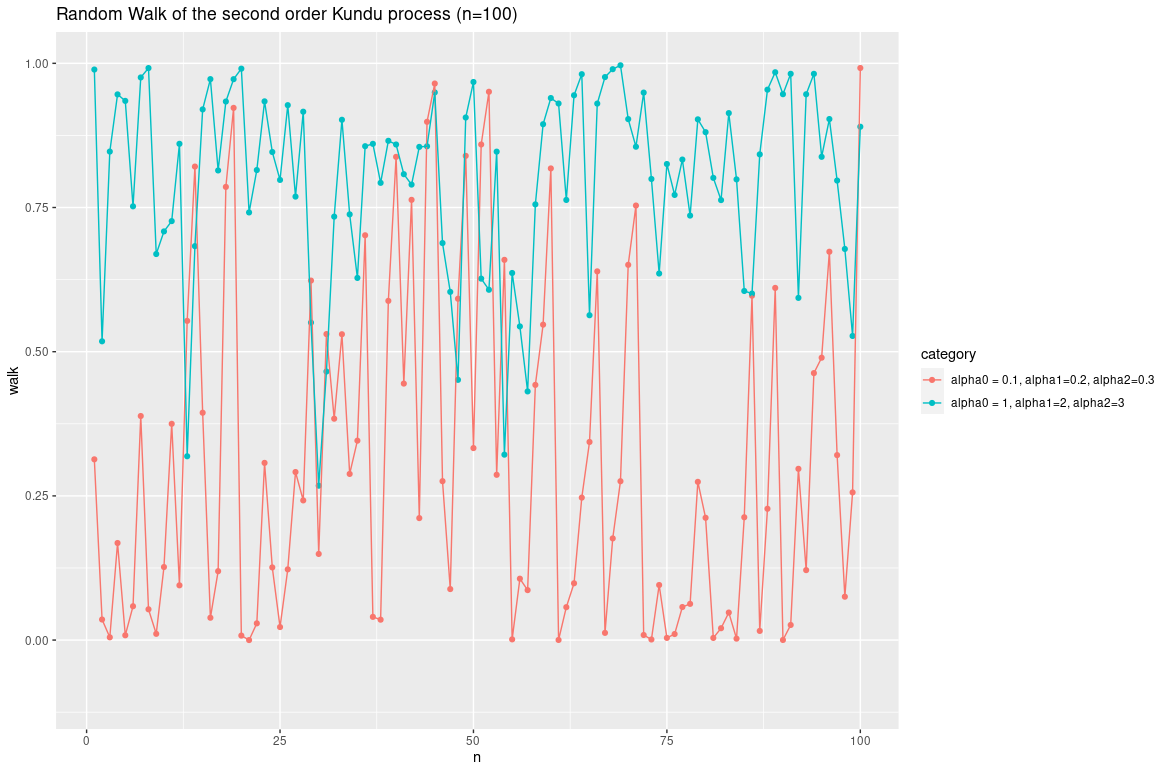}
 		\caption{$n=100$}  
 	\end{subfigure}
 	\begin{subfigure}{9cm}
 		\centering
 		\includegraphics[width=7cm]{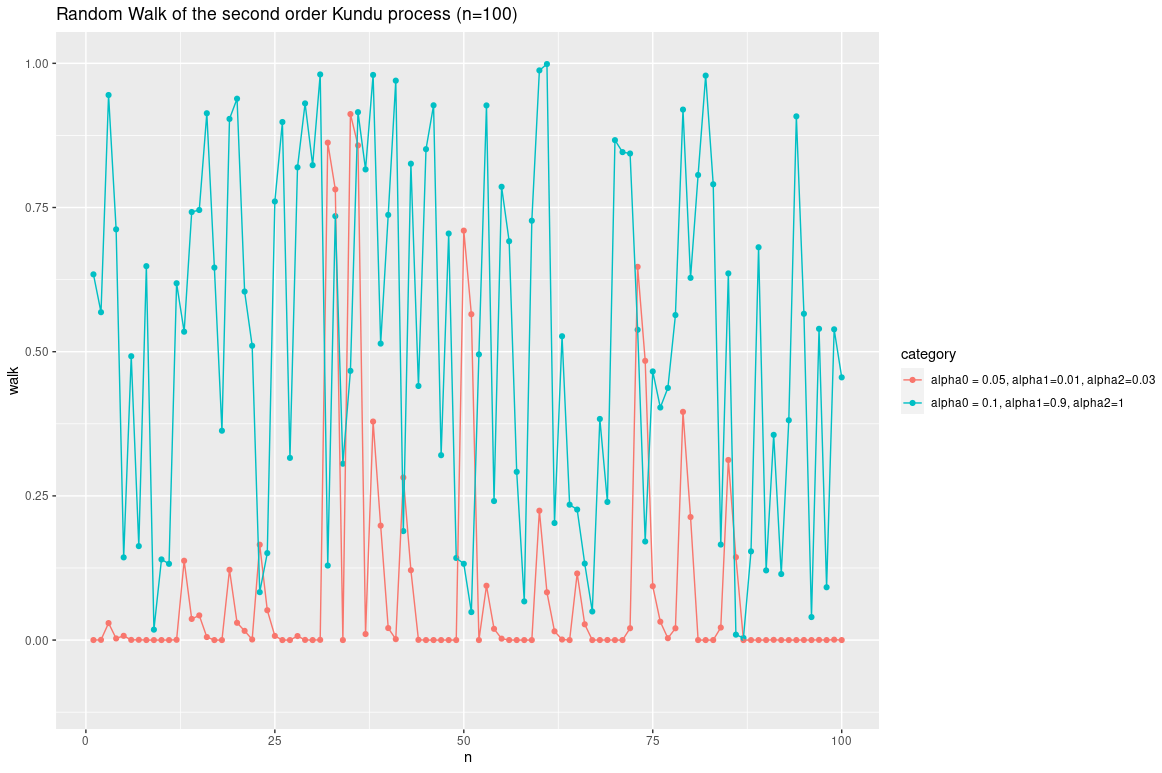}
 		\caption{$n=100$} 
 	\end{subfigure}
 	\caption{Sample Path with $100$ steps}
 	\label{2rfig7}
 \end{figure}
 
 \bigskip
 
 The joint distribution function of $(X_0,X_1)$ is given by
 \begin{equation}
 	F_{X_0,X_1}(x_0,x_1)= x_0^{\alpha_2}\min\{x_0^{\alpha_1},x_1^{\alpha_2} \} \min\{x_0^{\alpha_0},x_1^{\alpha_1}\}x_1^{\alpha_0} \ \ \ 0<x_0,x_1<1.
 \end{equation}
 
 We refer to  Figure \ref{2Rfig8} for the distribution function plots for different values of $\alpha_0$, $\alpha_1$ and $\alpha_2$.
 
 \begin{figure} [h!]
 	\begin{subfigure}{8cm}
 		\centering
 		\includegraphics[width=10cm]{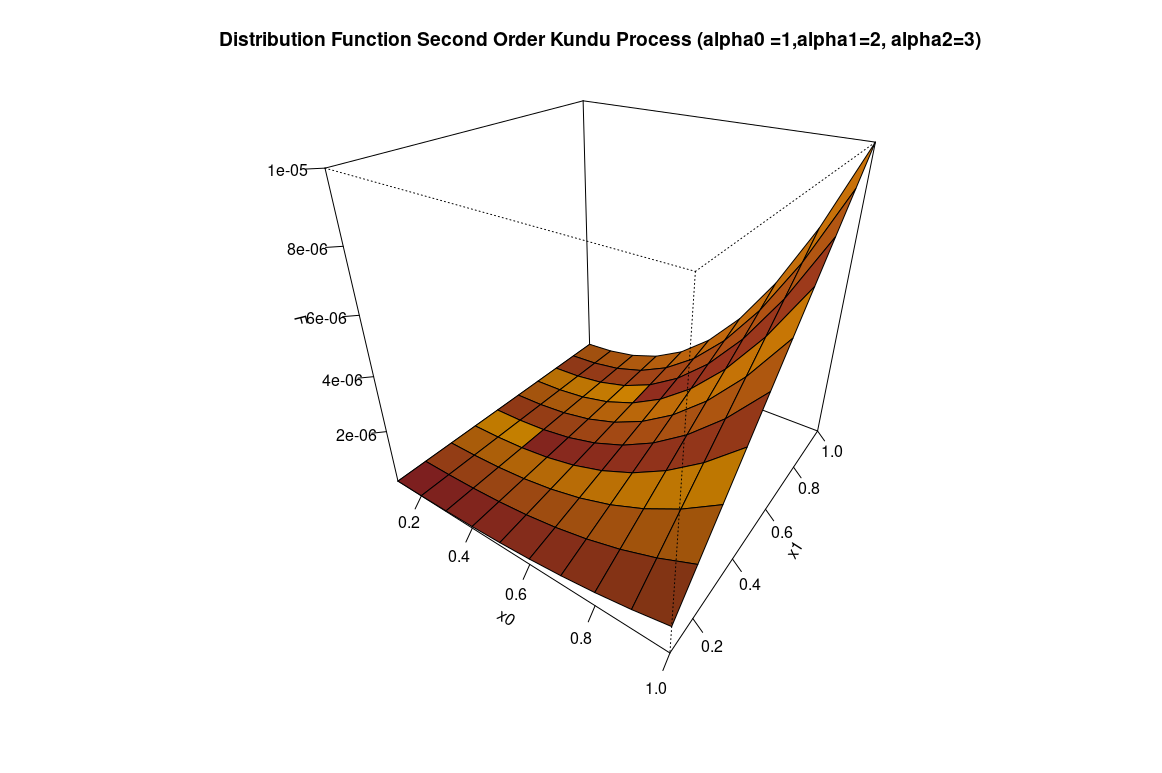}
 		\caption{$\alpha_0 =1$, $\alpha_1 =2$ and $\alpha_2 =3$} 
 	\end{subfigure}
 	\begin{subfigure}{8cm}
 		\centering
 		\includegraphics[width=10cm]{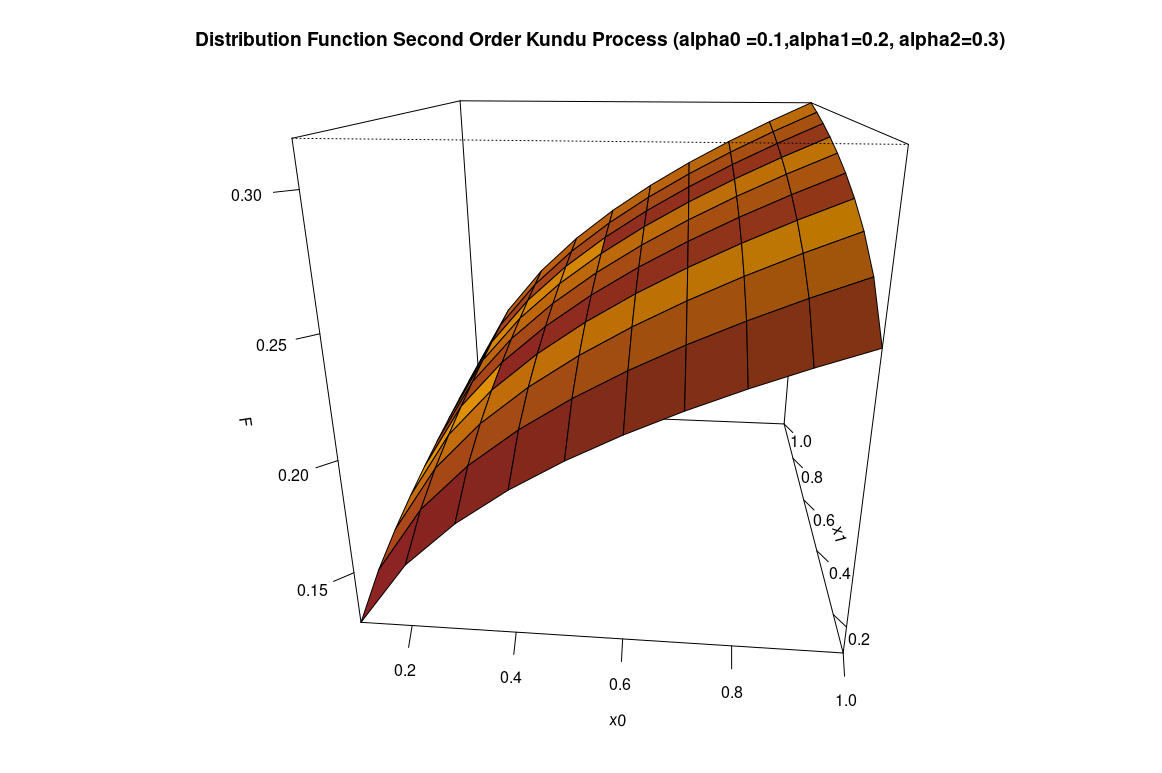}
 		\caption{$\alpha_0 =0.1$, $\alpha_1 =0.2$ and $\alpha_2 =0.3$} 
 	\end{subfigure}
 	\caption{Distribution function plots of second order Kundu process}
 	\label{2Rfig8}
 \end{figure}
 
 \bigskip
 
 \begin{remark}
 	Note that for the process defined in equation $(4.1)$  it will be difficult to derive an expression for $E(X_0X_1)$, consequently the auto-correlation function does not appear to have a simple expression. 
 \end{remark}
 
 \bigskip
 
 A reverse second order Kundu process would be defined by
 \begin{equation}\label{PFDP(2)-rev}
 	X_n=\max \left\{U_n^{1/\alpha_0},U_{n+1}^{1/\alpha_1},U_{n+2}^{\alpha_2} \right\},
 \end{equation}
 where $\alpha_0, \alpha_1, \alpha_2 >0$.  In this case also we have $X_n \sim PFD(\alpha_0+\alpha_1+\alpha_2).$
 
 \bigskip
 In parallel fashion one can define a $k$-th order Kundu process defined by assuming $U_{-k},...,U_{-1},U_{0},U_{1},U_{2}....$ are all i.i.d 
 $uniform(0,1)$ random variables and for $n=0,1,2,...$ defining
 \begin{equation} \label{PFDP(2)}
 	X_n= \max \left\{U_n^{1/\alpha_0},U_{n-1}^{1/\alpha_1},U_{n-2}^{1/\alpha_2},...,U_{n-k}^{1/\alpha_k} \right\},
 \end{equation}
 where $\alpha_0, \alpha_1, \alpha_2,...,\alpha_k >0$. Again, this is a stationary  non-Markovian process with each $X_n$ having $PFD(\alpha_0+...+\alpha_k)$ for $n=0,1,2,...$.
 
 \begin{proposition}
 	If the sequence of random variables $\{X_n\}$ is as defined in equation $(4.8)$, then the one-dimensional density of the process is 
 	\begin{eqnarray}
 		f(x_n) = \Big( \sum_{i=0}^{p} \alpha_i \Big) x^{(\sum_{i=0}^{p} \alpha_i)-1};   \mbox{      } 0\leq x_n \leq 1.
 	\end{eqnarray}
 \end{proposition}
 
 \begin{proposition}
 	If the sequence of random variables $\{X_n\}$ is as defined in equation $(4.8)$, then the mean and variance are 
 	\begin{eqnarray}
 		E(X_n) &=&  \frac{\sum_{i=0}^{p} \alpha_i}{\sum_{i=0}^{p} \alpha_i + 1}  \\
 		V(X_n) &=&  \frac{\sum_{i=0}^{p} \alpha_i}{(\sum_{i=0}^{p} \alpha_i +1)^2(\sum_{i=0}^{p} \alpha_i + 2)}.
 	\end{eqnarray}
 \end{proposition}
 \begin{proof}
 	These are trivial consequences of Proposition 3. 
 \end{proof}
 
 \bigskip
 
 Now, recall the alternative  stationary power function$(\alpha)$ process is defined by;
 
 $X_0=U_0^{1/\alpha} \sim PF(\alpha)$ and for $n=1,2,...$ 
 \begin{equation}\label{ARPFProcess}
 	X_n= \max\{X_{n-1}^{\alpha/(\alpha-\delta)},U_n^{1/\delta}\},
 \end{equation}
 where $\delta \in (0,\alpha)$ and the $U_n$'s are i.i.d. $uniform(0,1)$ random variables.
 
 It wouldl be desirable to construct a second order version of such a process based on a suitable sequence of i.i.d. $uniform(0,1)$ random variables. We would like to have $X_n \sim PFD(\alpha)$ for $n=0,1,2,...$ and to have
 
 \begin{equation}\label{2ARPFProcess}
 	X_n= \max\{X^\beta_{n-2}, X^\gamma_{n-1}, U^\delta_n \},
 \end{equation}
 for $n=0,1,2,...$, where $\beta,\gamma,\delta$ take on suitable values.
 \bigskip
 Analogously, one might comsider the possible existence of a $k$th-order $PFD(\alpha)$ process of the form:
 
 \begin{equation}\label{kARPFProcess}
 	X_n= \max\{X^{\delta_k}_{n-k},...,X^{\delta_1}_{n-1}, U^\delta_n \},
 \end{equation}
 for $n=0,1,2,...$, where each of $\delta_k,...,\delta_1, \delta$  take on suitable values. At the present moment we have been unable to determine the nature of the long run disributions of such second and $k$-th order Markov processes. It is unlikely that a limiting power function distribution will be associated with such higher order processe. In order for such higher order processes have power function distribution marginal distributions, the definitions in (\ref{2ARPFProcess}) and (\ref{kARPFProcess}) will need to be modified. At present, how to achieve this remains unresolved.

 \section{Statistical Inference}
 \subsection{Estimation}
 In the following we derive consistent estimator of the parameters of the processes defined in  $(2.1)$ and $(3.1)$.  Also, it is clear from the definitions of the  two processes that bivariate joint density do not exists and all likelihood based inferences or Bayesian inference cannot be considered for estimating the parameters. In the following we use the method of moments to obtain consistent estimators for the process parameters.
 
 \subsubsection{Method of moments}
 For the observed sample paths $\{x_1,...,x_m\}$, define the following statistics
 \begin{eqnarray}
 	\bar{X}  &=& \frac{1}{m}\sum_{i=1}^{m} X_i  \\
 	P &=& \frac{1}{m} \sum_{i=1}^{m}  I(X_i <X_{i-1}).
 \end{eqnarray}
 where $I(.)$ is the indicator function. 
 \bigskip 
 For the given process in $(2.1)$ and $P(X_1 <X_2)$ for the case $\alpha > \beta$.  The following  two moment equations can be solved for estimates of $\alpha$ and $\beta$.
 \begin{eqnarray}
 	\overline{X}&=& \frac{\alpha +\beta}{\alpha + \beta +1}
 	\\
 	P&=&\frac{\alpha}{2\alpha+\beta}.
 \end{eqnarray}
  
 We obtain the following consistent estimates
 
 \begin{eqnarray}
 	\hat{\alpha}&=& \frac{P \overline{X}}{P \overline{X}- \overline{X} - P +1}
 	\\
 	\hat{\beta}&=& \frac{\hat{\alpha} - 2 \hat{\alpha} P}{P}.
 \end{eqnarray}
 
 This solution does not always satisfy the condition $\alpha > \beta$. Now, for the case $\alpha < \beta$  we have
 
 $$P(X_1<X_2)=\frac{\beta+\alpha}{2\beta+\alpha}.$$
 
 The following  two moment equations can be solved for estimates of $\alpha$ and $\beta$.
 \begin{eqnarray}
 	\overline{X}&=& \frac{\alpha +\beta}{\alpha + \beta +1}
 	\\
 	P&=&\frac{\beta+\alpha}{2\beta+\alpha}.
 \end{eqnarray}
 
 We can then obtain the following consistent estimates
 
 \begin{eqnarray}
 	\hat{\alpha}&=&  \frac{\overline{X}(1-2P)}{P(\overline{X}-1)}
 	\\
 	\hat{\beta}&=& \frac{ \hat{\alpha}(1-P)}{(2P-1)}.
 \end{eqnarray}
 
 In the following we derive consistent estimator of the process given in $(3.1)$.  Recalling equation $(3.9)$ the first population moment of the process in $(3.1)$ is
 \begin{eqnarray*}
 	E(X_n) = \frac{\alpha}{\alpha +1}
 \end{eqnarray*}
 and also 
 \begin{eqnarray*}
 	P(X_1 < X_0) = \frac{\alpha}{\alpha + \delta}.
 \end{eqnarray*}
 For the observed sample points $\{x_1,...,x_m\}$, equating the above two expressions to their with their  sample versionss, we have the folowinge two equations to solve.
 \begin{eqnarray}
 	\frac{\alpha}{\alpha +1} &=&  \bar{X} \\
 	\frac{\alpha}{\alpha +\delta} &=&  {P} .
 \end{eqnarray}
 
 Solving these equation will give estimates for $\alpha$ and $\delta$, i.e.,
 \begin{eqnarray}
 	\hat{\alpha}  &=& \frac{\bar{X}}{1- \bar{X}}  \\
 	\hat{\delta}  &=& \frac{\bar{X} - \bar{X} {P} }{{P}-\bar{X} {P}}.
 \end{eqnarray}

 \section{Applications}
 In the following two subsections, we illustrate a simulation study and give examples
 of real-life applications of the two power processes given in $(2.1)$ and $(3.1)$.
 
 \subsection{Simulation study}
 For the processes given in equation $(2.1)$ \& $(3.1)$ and for the same  moment estimators given in Section 4, we illustrate a simulation procedure on the behaviour of estimators by varying the parameter values with increasing sample path sizes.  
 
 We have simulated $2000$ data sets of sample path size  $m = 20, 30, 50, 100, 200, 500$ for the parameter  vectors $(\alpha=0.5,\beta=0.1)$,$(\alpha=1,\beta=2)$ and $(\alpha=2,\beta=1)$.  Similarlly,  for the same number of data sets with the same varying sample path sizes, we have simulated observations from the process in $(3.1)$ for the parameter vectors  
 $(\alpha=0.5, \delta =0.1)$, $(\alpha=4, \delta =2)$ and $(\alpha=2, \delta =0.5)$.

 We refer to the following Figures \ref{bfig12}--\ref{brfig13} for the bootstrapped distribution of each of the parameter estimates for the process in $(2.1)$. The numerical evidence suggests that as step size increases, moment estimates approach the true parameter values for both $\alpha$  and $\beta$ with standard errors that are decreasing as the step size increases.

 We refer to the following Figures \ref{b3.1fig12}--\ref{b3.1fig14} for the bootstrapped distribution of each of the parameter estimates for the process in $(3.1)$. The numerical evidence suggests that as sample step size increases, moment estimates approach the true parameter values for $\alpha$ with standard errors that are decreasing as the step size increases. Also, for the true parameter values for $\delta$ with moderate step sizes  convering to the true value and with decreasing in standard errors. Also, notice that when $0<\alpha<1$ both the estimates shows marginal bias but for step sizes greater than $100$ converges to the true value.

 \begin{figure} [h!]
 	\begin{subfigure}{9cm}
 		\centering
 		\includegraphics[width=8cm]{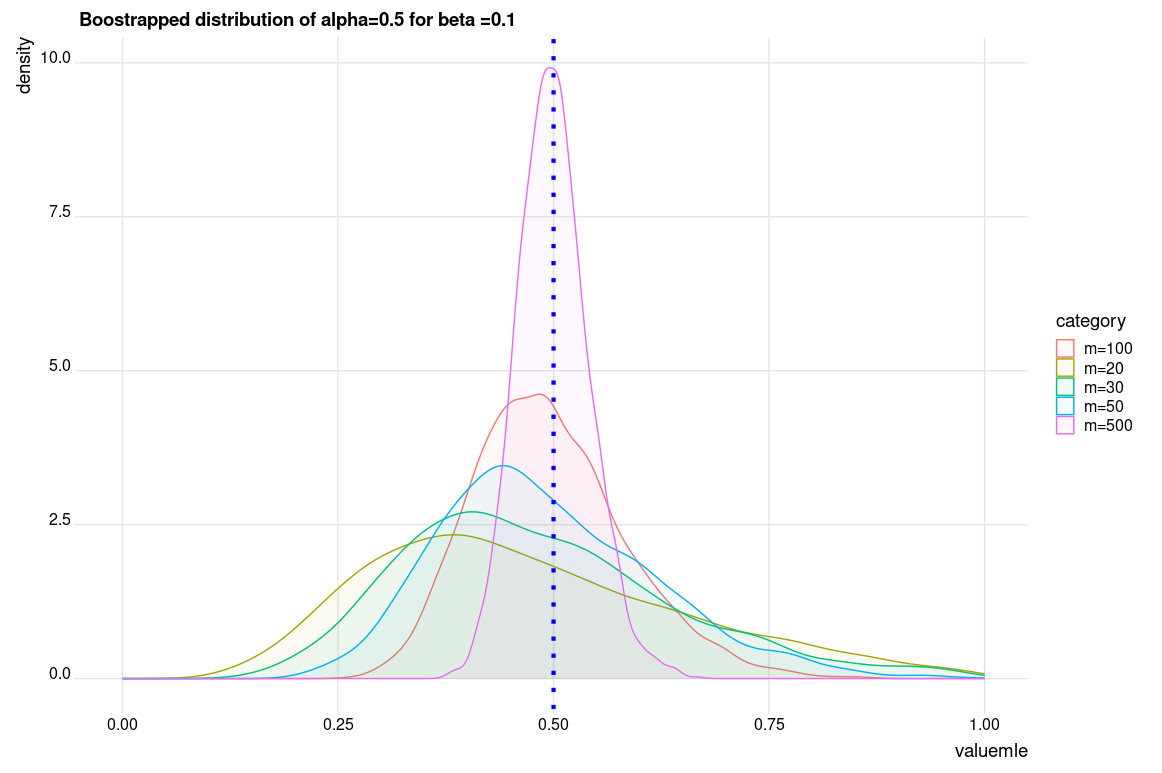}
 		\caption{Distribution of $\alpha = 0.5$ for $\beta =0.1$} 
 	\end{subfigure}
 	\begin{subfigure}{9cm}
 		\centering
 		\includegraphics[width=8cm]{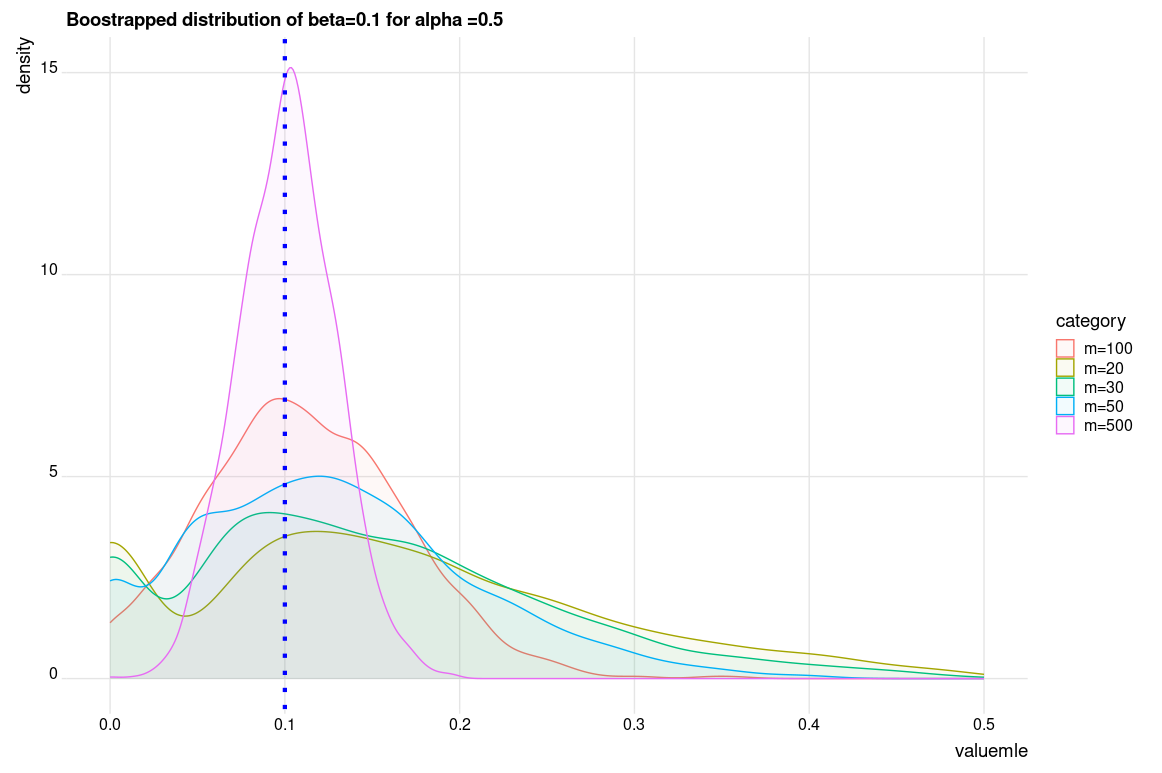}
 		\caption{Distribution of $\beta = 0.1$ for $\alpha =0.5$} 
 	\end{subfigure}
 	\caption{Boostrapped Distributions for the process $(2.1)$}
 	\label{bfig12}
 \end{figure}
 
 \begin{figure} 
 	\begin{subfigure}{9cm}
 		\centering
 		\includegraphics[width=8cm]{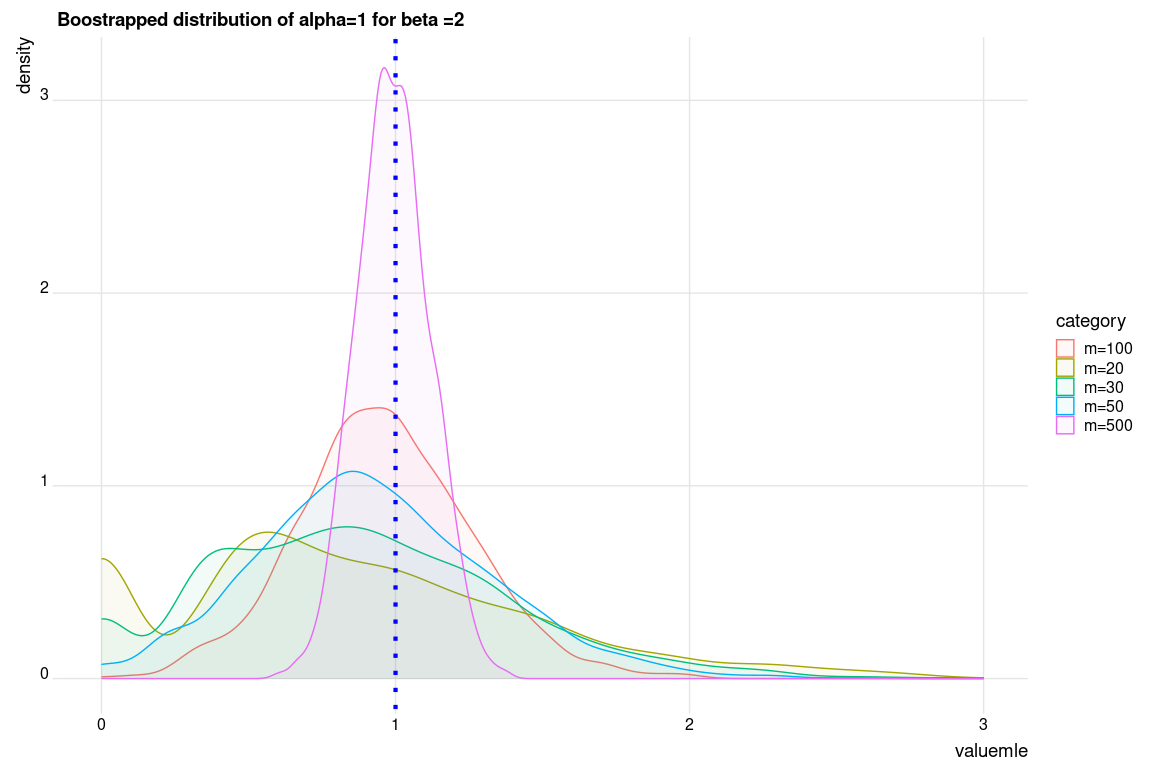}
 		\caption{Distribution of $\alpha =1$ for $\beta =2$} 
 	\end{subfigure}
 	\begin{subfigure}{9cm}
 		\centering
 		\includegraphics[width=8cm]{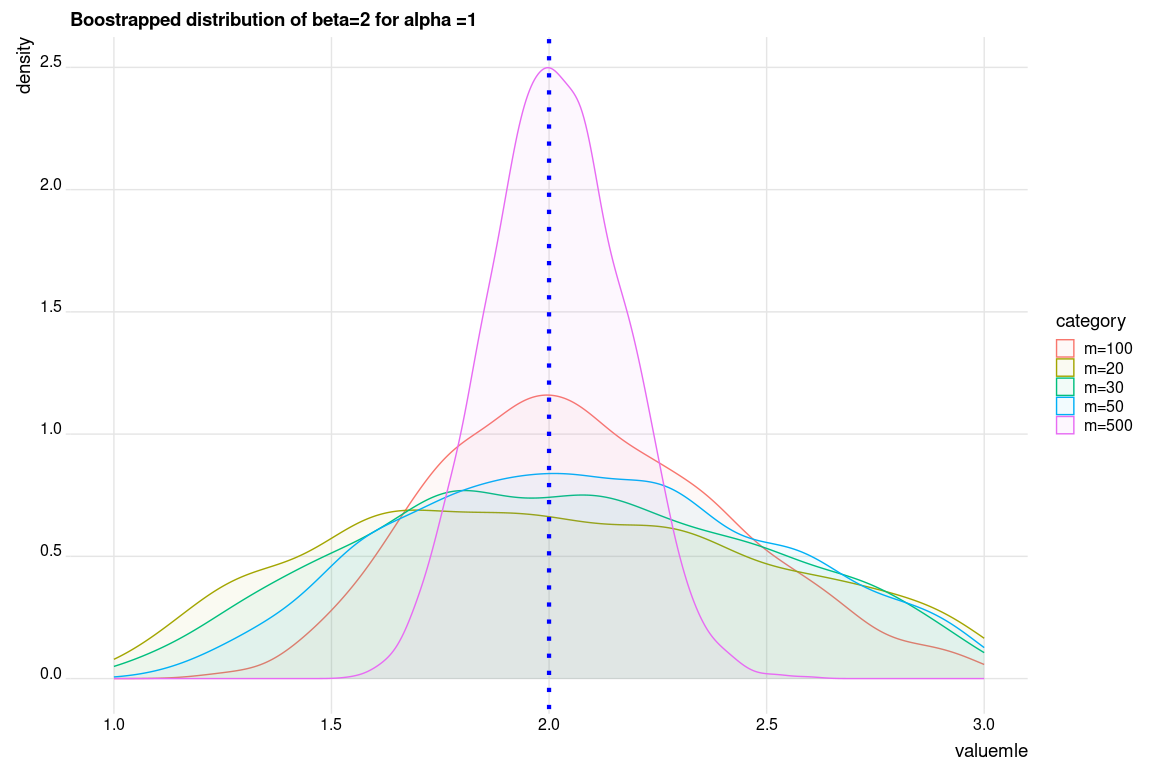}
 		\caption{Distribution of $\beta = 2$ for $\alpha =1$} 
 	\end{subfigure}
 	\caption{Boostrapped Distributions for the process $(2.1)$}
 	\label{fig13}
 \end{figure}
 
 \begin{figure} 
 	\begin{subfigure}{9cm}
 		\centering
 		\includegraphics[width=8cm]{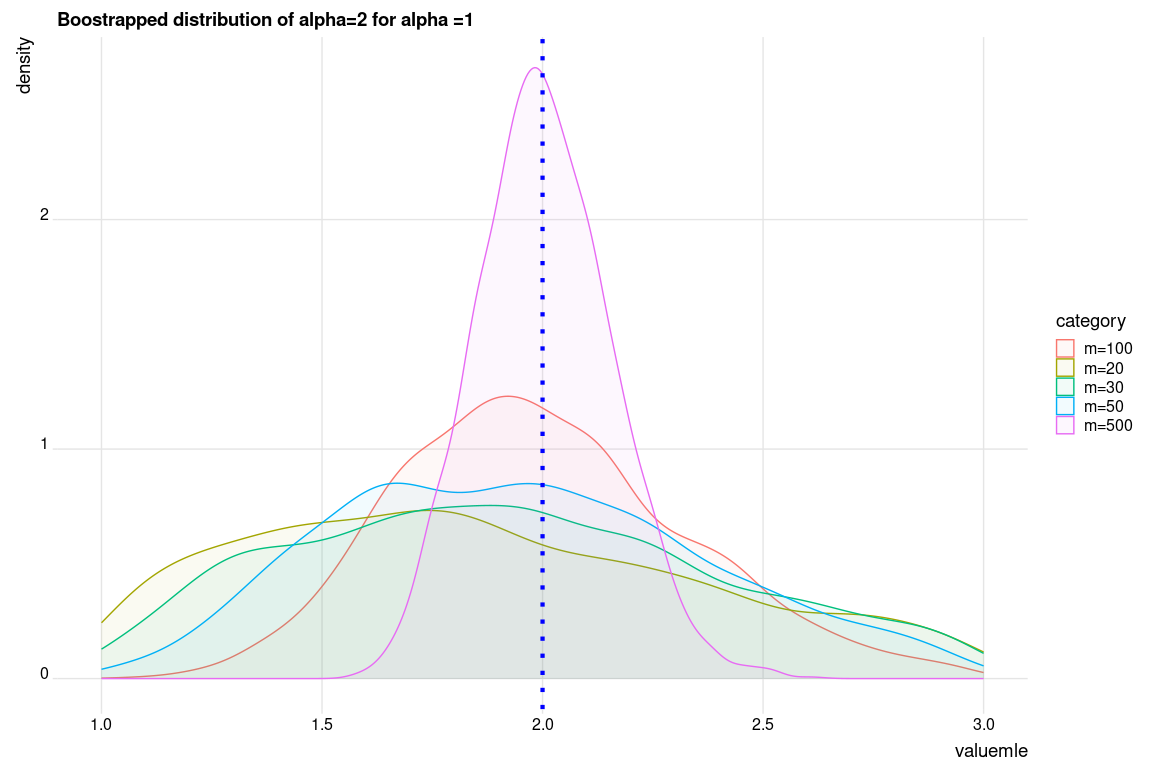}
 		\caption{Distribution of $\alpha =2$ for $\beta =1$} 
 	\end{subfigure}
 	\begin{subfigure}{9cm}
 		\centering
 		\includegraphics[width=8cm]{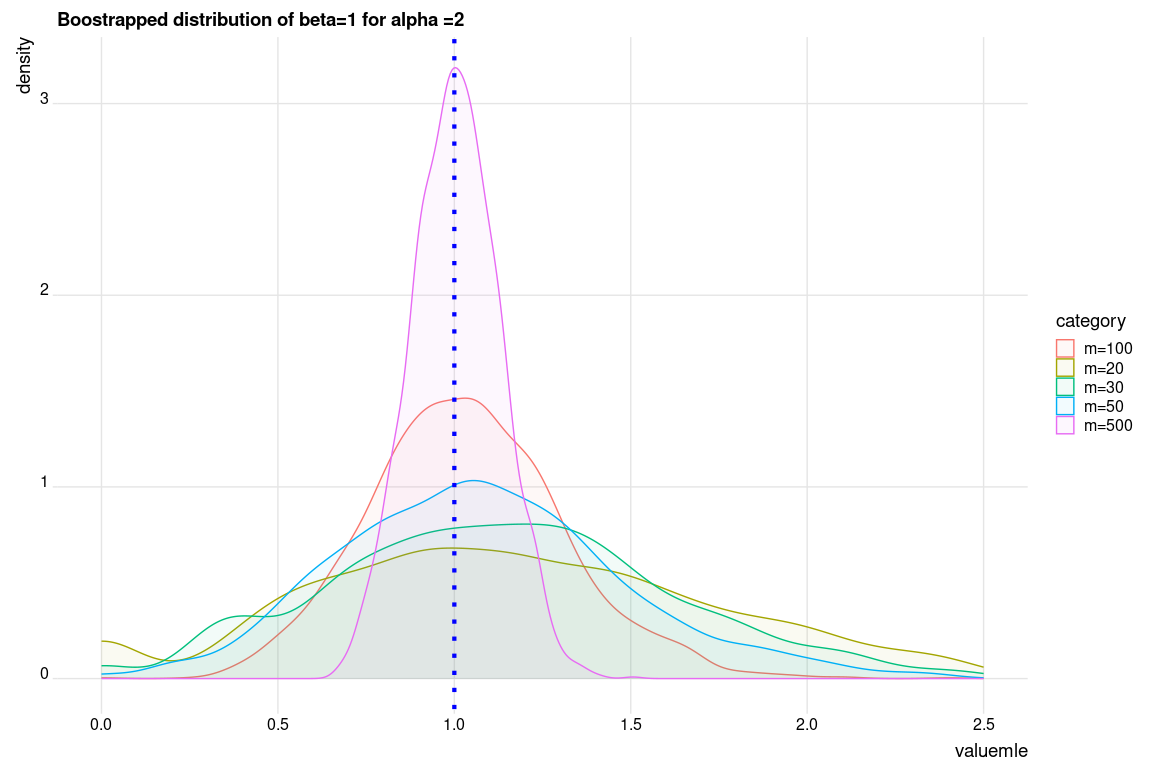}
 		\caption{Distribution of $\beta = 1$ for $\alpha =2$} 
 	\end{subfigure}
 	\caption{Boostrapped Distributions for the process $(2.1)$}
 	\label{brfig13}
 \end{figure}

 \bigskip

 \begin{figure} [h!]
 	\begin{subfigure}{9cm}
 		\centering
 		\includegraphics[width=8cm]{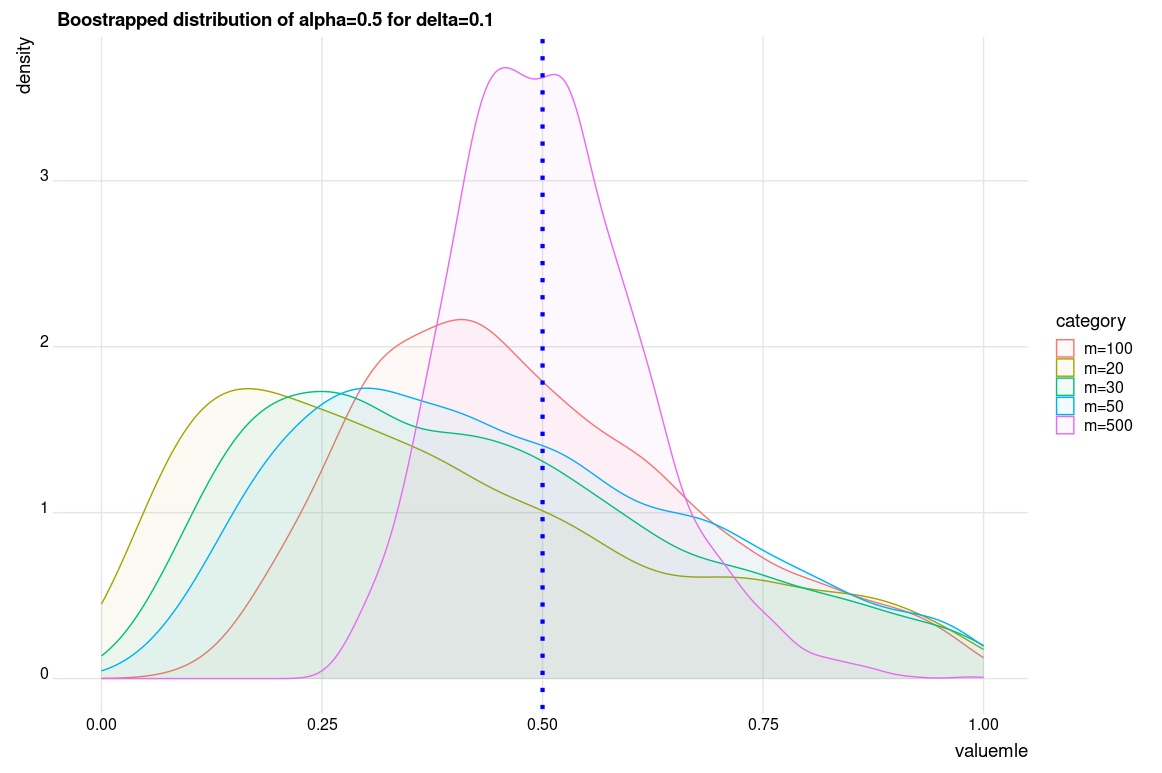}
 		\caption{Distribution of $\alpha =0.5$ for $\delta =0.1$} 
 	\end{subfigure}
 	\begin{subfigure}{9cm}
 		\centering
 		\includegraphics[width=8cm]{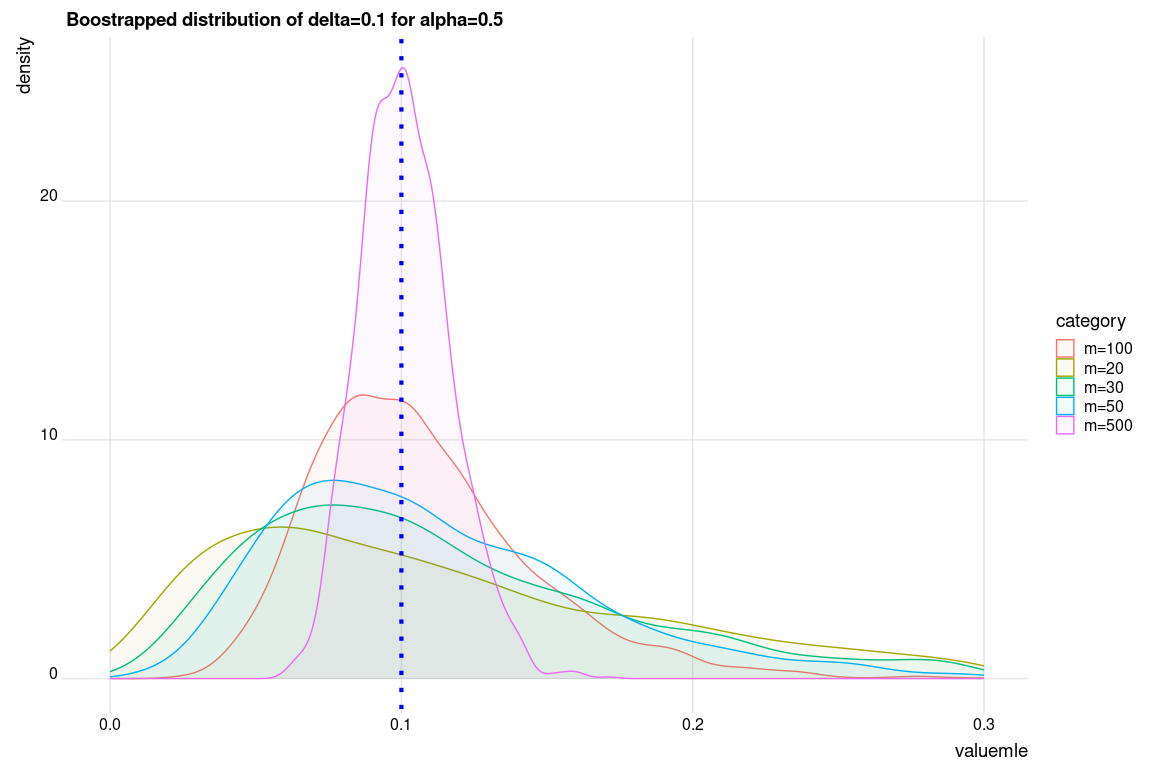}
 		\caption{Distribution of $\delta = 0.1$ for $\alpha =0.5$} 
 	\end{subfigure}
 	\caption{Boostrapped Distributions for the process $(2.1)$}
 	\label{brfig13}
 	\label{b3.1fig12}
 \end{figure}

\begin{figure} 
\begin{subfigure}{9cm}
 	\caption{Boostrapped Distributions for the process $(3.1)$}
  		\centering
 		\includegraphics[width=8cm]{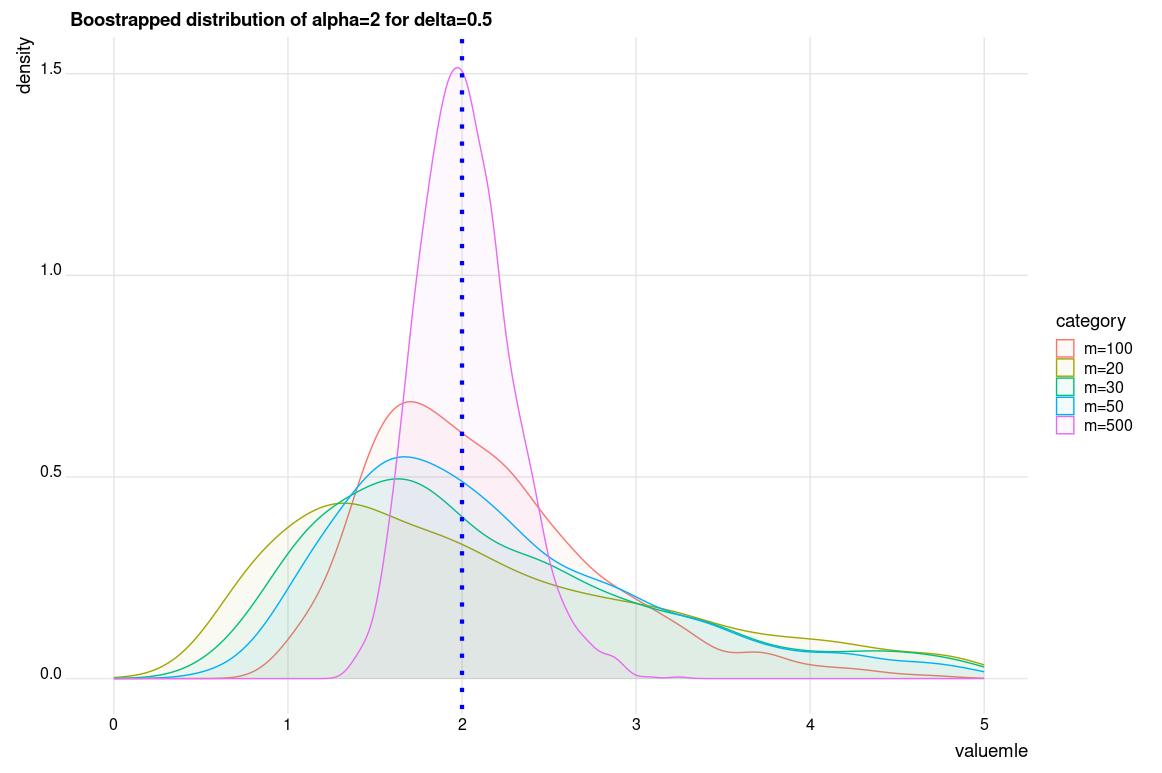}
 		\caption{Distribution of $\alpha =2$ for $\delta =0.5$} 
 	\end{subfigure}
 	\begin{subfigure}{9cm}
 		\centering
 		\includegraphics[width=8cm]{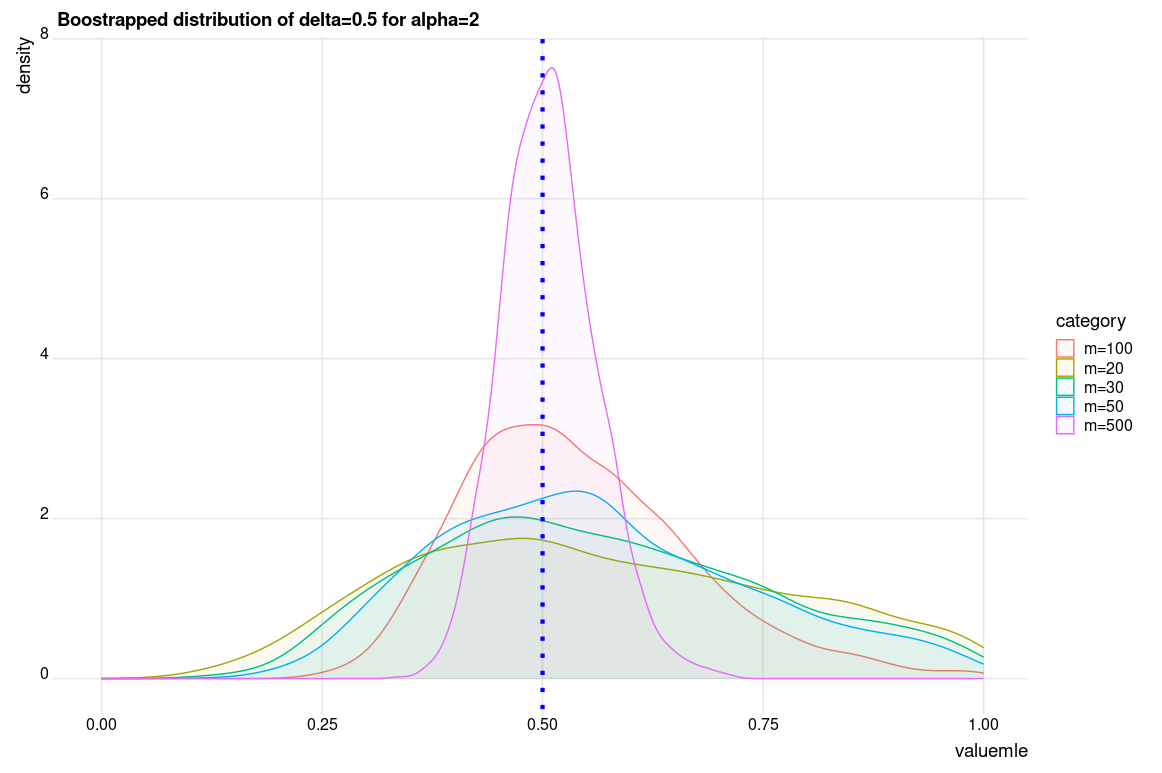}
 		\caption{Distribution of $\delta = 0.5$ for $\alpha =2$} 
 	\end{subfigure}
 	\caption{Boostrapped Distributions for the process $(3.1)$}
 	\label{fig13}
 \end{figure}

 \begin{figure}
 	\begin{subfigure}{9cm}
 		\centering
 		\includegraphics[width=8cm]{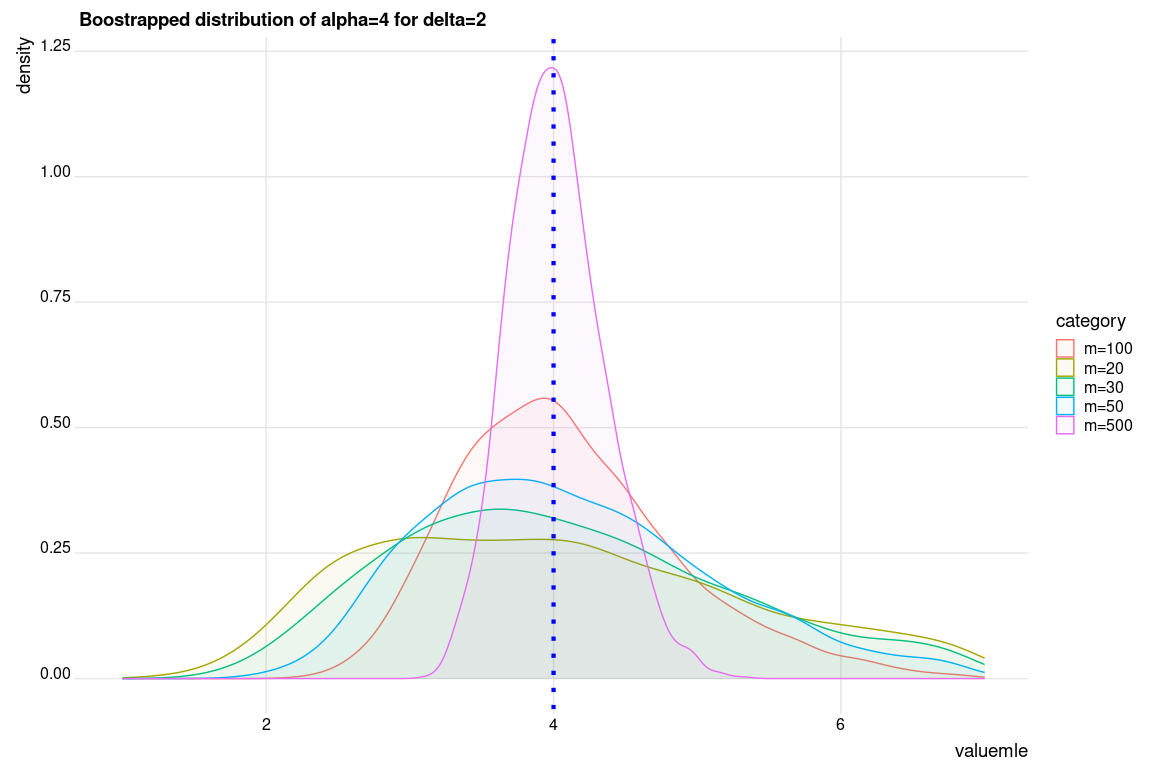}
 		\caption{Distribution of $\alpha =4$ for $\delta =2$} 
 	\end{subfigure}
 	\begin{subfigure}{9cm}
 		\centering
 		\includegraphics[width=8cm]{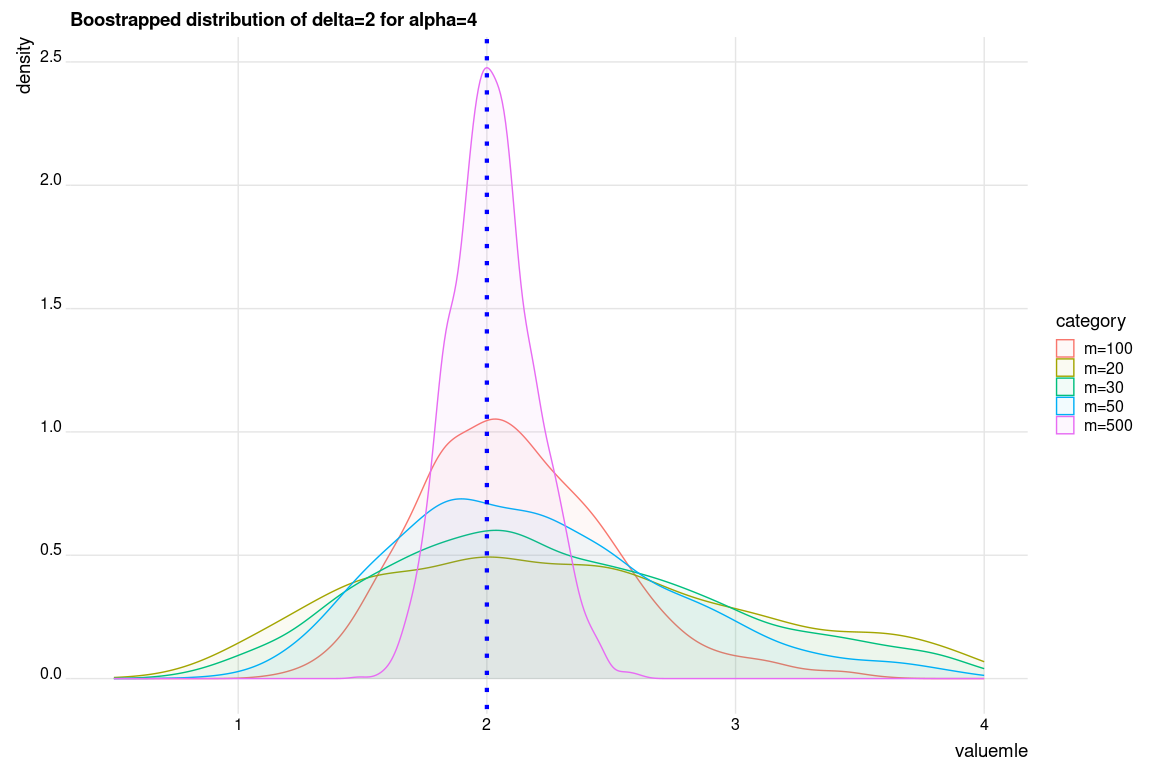}
 		\caption{Distribution of $\delta = 2$ for $\alpha =4$} 
 	\end{subfigure}
 	\caption{Boostrapped Distributions for the process $(3.1)$}
 	\label{b3.1fig14}
 \end{figure}
 
 \subsection{Real-life data}
 In the following we consider a Gold-price data, which is also mentioned in Kundu \cite{kd22}. We have considered a data set from 23 March 2023 to 24 April 2023, which is 30 days Gold price per gram in Indian.  
 
 Here we assume that the data generating process for the Gold price  having PRH process with unknown marginal distribution $F^*$ and underlying PFD process of either in $(2.1)$ or $(3.1)$. We used a non-parametric approach to estimate $F^*$ by the $F_m$ the empirical distribution function of the $m$ available data points. By using the $F_m$  transformation on the available Gold price data will be a sample froma PFD process.
 We refer to Figure \ref{gold1} for the Gold price of 90 days and its marginal transformation to have PFD process. 
We refer to Table \ref{table:1} and \ref{table:2} for the fitted values for the two process. Finally, by using mean square error mesaure distance between actula and the process in $(2.1)$ is $0.213$. Similarly, with the same measure distance between actual and the process in $(3.1)$ is $0.158$. Hence, we conclude that PFD process defined in equation $(3.1)$ fits the Gold price data better than Kundu process.

\begin{figure} [h!]
	\begin{subfigure}{9cm}
		\centering
		\includegraphics[width=8cm]{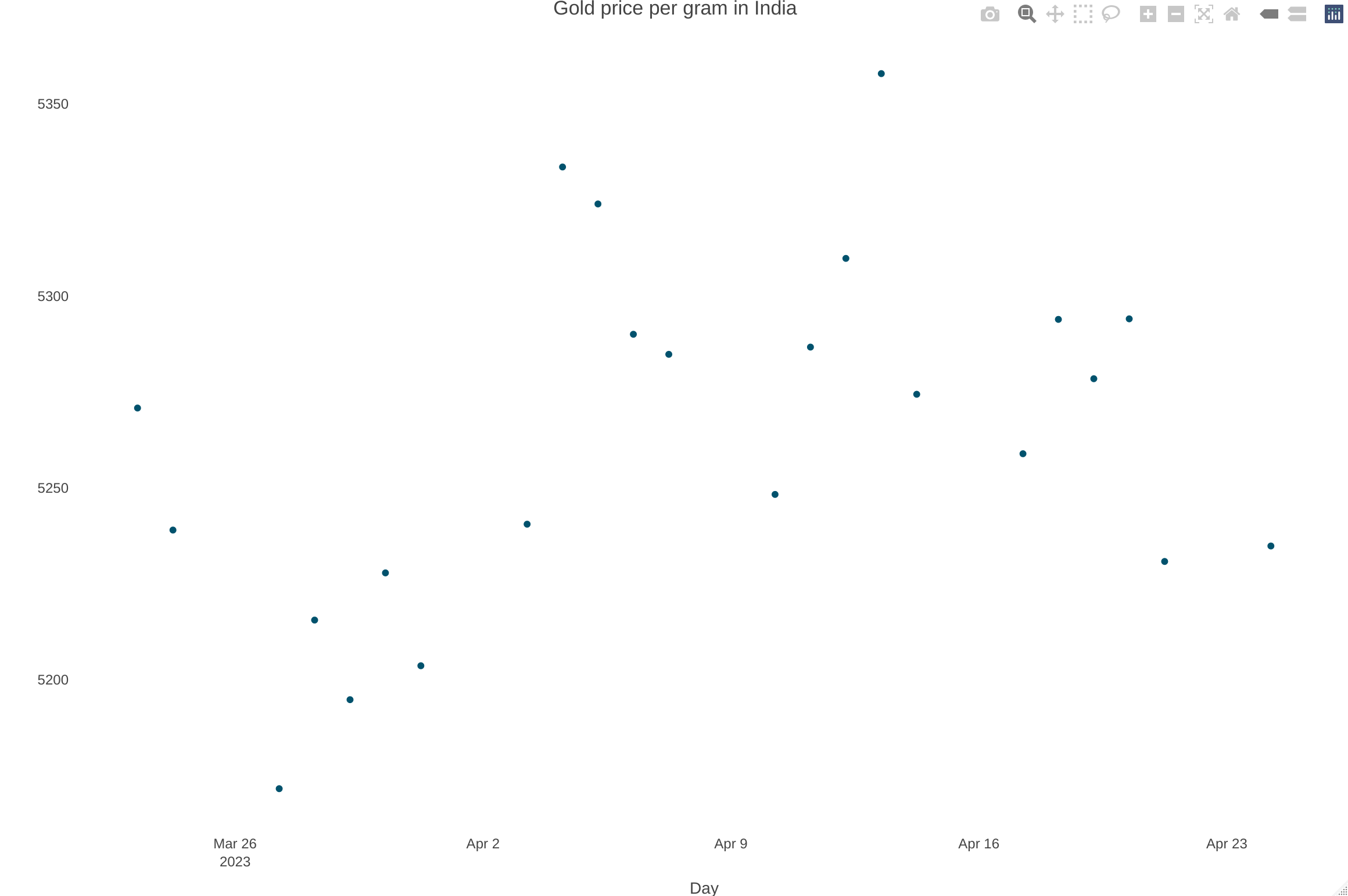}
		\caption{Gold price data of 30 days} 
	\end{subfigure}
	\begin{subfigure}{9cm}
		\centering
		\includegraphics[width=8cm]{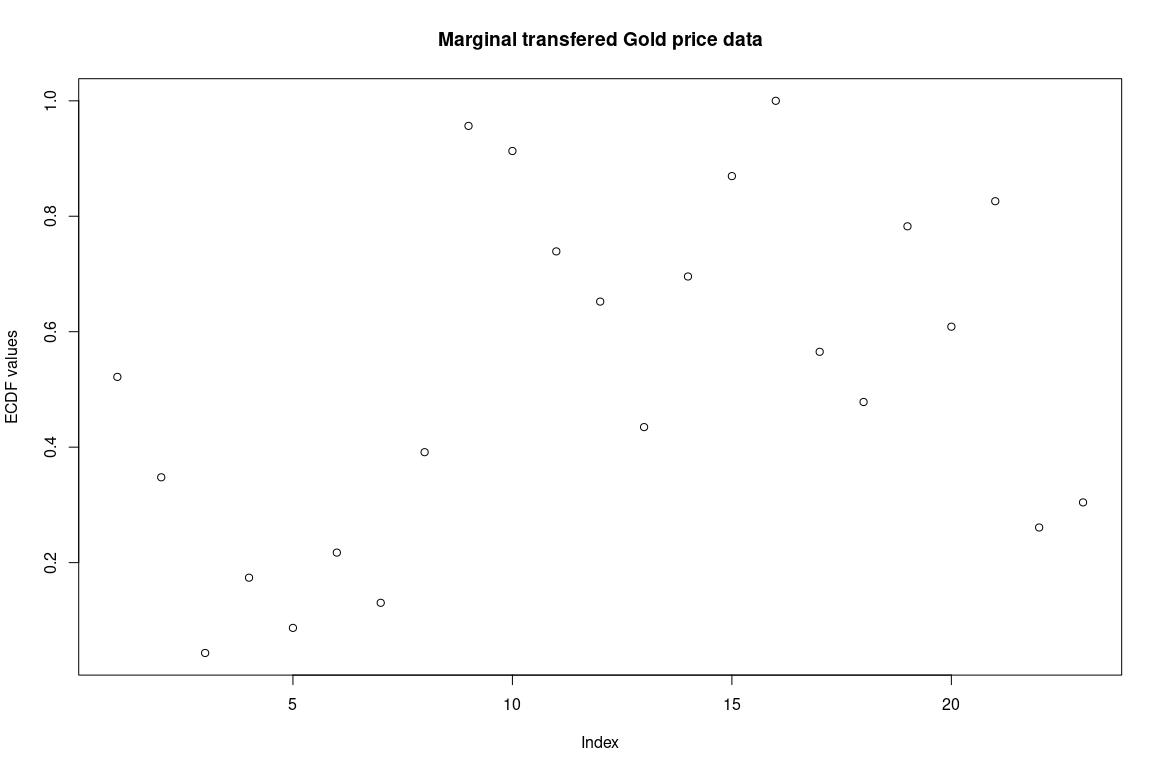}
		\caption{Marginal Transfered data} 
	\end{subfigure}
\caption{Gold price data analysis} 
	\label{gold1}
\end{figure}

\begin{table}[h!]
	\centering
	\begin{tabular}{||c | c ||} 
		\hline
		Parameters & Estimates \\ [0.5ex] 
		\hline\hline
		$\alpha$ & $0.839$ \\ 	\hline
		$\beta$ & $0.251$ \\ 	\hline
		auto-correlation & $-0.780$\\  [1ex] 
		\hline
	\end{tabular}
	\caption{Gold price data fitted to $(2.1)$ process}
	\label{table:1}
\end{table}

\begin{table}[h!]
	\centering
	\begin{tabular}{||c | c ||} 
		\hline
		Parameters & Estimates \\ [0.5ex] 
		\hline\hline
		$\alpha$ & $1.091$ \\ 	\hline
		$\delta$ &  $1.000$\\ 	\hline
		auto-correlation & $0.118$ \\  [1ex] 
		\hline
	\end{tabular}
	\caption{Gold price data fitted to $(3.1)$ process}
	\label{table:2}
\end{table}

\section{Conclusion}
By considering power function distribution process, which is a generalization of all the PRH processes with suitable marginal transformation. 
We have considered two PFD processes: stationary non-Markovian (called Kundu process) and a new process with stationary Markovian properties. We have derived the marginal and bivariate distributional properties and included moment computation up to the auto-correlation function of the two processes. Since the derived bivariate distribution does not have a density, likelihood-based inferences such as maximum likelihood and the Bayesian approach cannot be used for estimating the parameters of the process. However, we have derived one more property of the process and obtained a simple expression to obtain consistent estimators of the parameters of the processes. We also included a simulation study of the estimators with varying sample step sizes and illustrated the behaviour of the estimators. We have also included an example of an application to the Gold price per gram data. Finally, we also had a short note on the higher-order version of the two processes. As mentioned earlier, we can transform our process marginally to be a PRH process, and that consideration of such processes will be the subject of a further report.

\section*{Acknowledgment}
The second author's research was sponsored by the  DST Inspir. The author Manjunath's research was sponsored by the Institution of Eminence (IoE),
University of Hyderabad (UoH-IoE-RC2-21-013)

\end{document}